\documentclass[11pt,a4paper]{article}
\usepackage[margin=1in]{geometry}
\usepackage[latin1]{inputenc}
\usepackage{amsmath}
\usepackage{amsthm}
\usepackage{amsfonts}
\usepackage{amssymb}
\usepackage[makeroom]{cancel}
\usepackage{graphicx}
\usepackage{epstopdf}
\usepackage{bm}
\usepackage{fullpage}
\usepackage{color}
\usepackage{hyperref}
\usepackage[noblocks]{authblk}
\usepackage{caption,subcaption}
\usepackage{float}
\usepackage{wrapfig}
\usepackage{caption}
\usepackage{subcaption}
\usepackage{cite}
\usepackage{cleveref}
\usepackage{cases}
\usepackage{enumitem}
\usepackage{ulem}
\usepackage{tikz}
\usepackage{ifthen}
\usepackage[mathlines]{lineno}
\DeclareMathSizes{10}{10}{10}{10}

\hypersetup{%
    colorlinks=true,
        linkcolor=blue,
    filecolor=magenta,      
    urlcolor=violet,
    citecolor=blue,
}

\title{Well-posedness and qualitative properties of quasilinear degenerate evolution systems}
\author[1]{K. Mitra\footnote{email: \href{mailto:koondanibha.mitra@uhasselt.be}{koondanibha.mitra@uhasselt.be} }}
\author[2]{S. Sonner}
\affil[1]{Faculty of Science, Hasselt University, Hasselt, Belgium}
\affil[2]{Faculty of Science, Radboud University, Nijmegen, The Netherlands}
\numberwithin{equation}{section}
\makeatletter
\DeclareMathSizes{\@xpt}{\@xpt}{6}{5}
\makeatother

\newcounter{def}
\newcounter{assume}
 \stepcounter{assume}
 
 \newcounter{EUassumption}
 \stepcounter{EUassumption}
  \newcounter{EUXproperties}
 \stepcounter{EUXproperties}
 
\newtheorem{theorem}{Theorem}[section]
\newtheorem{lemma}{Lemma}[section]
\newtheorem{proposition}{Proposition}[section]
\newtheorem{corollary}{Corollary}[theorem]
\newtheorem{remark}{Remark}[section]
\newtheorem{definition}[def]{Definition}

\theoremstyle{definition}
\newtheorem{example}{Example}[section]
\newtheorem{assumption}{Assumption}[section]
\def \a  {\alpha}

\def \b  {\beta}
\def \g  {\gamma}
\def \G  {\Gamma}
\def \d  {\delta}
\def \D  {\Delta}
\def \e  {\varepsilon}
\def \eps {\epsilon}
\def \f  {\varphi}
\def \j {\ell}
\def \vr  {\varrho}

\def \Om {\Omega}

\def \t  {\tau}

\def \z  {\zeta}
\def \del {\nabla}
\def \Cf {\mathfrak{C}}

\def \p  {\partial}
\def \N  {{\mathbb{N}}}
\def \R  {{\mathbb{R}}}
\def \H  {{\cal H}^1}
\def \Hm  {{\cal H}^{-1}}

\def \W {{\bf {\cal W}}}
\def \X {{\bf {\cal X}}}
\def \Y {{\bf {\cal Y}}}
\def \Z {{\bf {\cal Z}}}

\def \calA {\mathfrak{A}}
\def \calB {\mathfrak{B}}
\def \calF {\mathfrak{F}}
\def \dd {\mathrm{d}}

\let\oldequation\equation
\let\oldendequation\endequation

\renewenvironment{equation}
  {\linenomathNonumbers\oldequation}
  {\oldendequation\endlinenomath}
  
  \let\oldalign\align
\let\oldendalign\endalign

\renewenvironment{align}
  {\linenomathNonumbers\oldalign}
  {\oldendalign\endlinenomath}

\begin{document}
\maketitle

\begin{abstract}
    We analyze nonlinear degenerate coupled PDE-PDE and PDE-ODE systems  that arise, for example, in the modelling of biofilm growth. One of the equations, describing the evolution of a biomass density, exhibits degenerate and singular diffusion. The other equations are either of advection-reaction-diffusion type or ordinary differential equations. Under very general assumptions the existence of weak solutions is proven by considering regularized  systems, deriving uniform bounds and using fixed point arguments. Assuming additional structural assumptions we also prove the uniqueness of solutions. 
    
    Global-in-time well-posedness is established for Dirichlet and mixed boundary conditions, whereas, only local well-posedness can be shown for homogeneous Neumann boundary conditions. Using a suitable barrier function and comparison theorems we formulate sufficient conditions for finite-time blow-up or uniform boundedness of solutions. Finally, we show that solutions of the degenerate parabolic equation inherit additional global spatial regularity if the diffusion coefficient has a power-law growth.
\end{abstract}

\textbf{Keywords:}  degenerate diffusion $\bullet$  biofilm models $\bullet$ quasilinear parabolic systems $\bullet$ PDE-ODE systems $\bullet$  well-posedness $\bullet$  regularity $\bullet$  finite time blow up  

\textbf{MSC:} 35K65, 35K59, 35A01, 35A02, 35B44, 35B45, 35B50

\tableofcontents

\section{Introduction}
This paper investigates the well-posedness and qualitative properties of weak solutions of a wide class of quasilinear parabolic systems where one of the equations shows degenerate and singular diffusion. 
We also consider couplings of such degenerate parabolic equations with ordinary differential equations (ODEs). 
The motivation for our work is models describing the growth of spatially heterogeneous biofilms in dependence of growth limiting substrates. The models are either formulated as systems of  partial differential equations (PDEs) or as coupled PDE-ODE systems, e.g. see \cite{eberl2001new,eberl2017spatially}.
Their characteristic and challenging features are the degenerate and singular diffusion effects in the equation for the biomass density and the nonlinear coupling of this equation to additional ODEs and/or PDEs for the substrates. 

Let $\Om\subset \R^d,\, d\in\N,$ be a bounded Lipschitz domain and $T>0$. 
We denote the parabolic cylinder by $Q:=\Om\times (0,T]$. Throughout this study, for a fixed $k\in \N$, $j\in \{1,\dots,k\}$ will denote an integer, and $\vec{w}=(w_1,\dots,w_k)$ a $k$-dimensional vector. We consider the following problem in $Q$,
\begin{subequations}\label{eq:main}
\begin{align} 
&\p_t M =\del\cdot[D_0(M)\del M] + f_0(M,\vec{S}),\label{eq:M}\\
&\p_t S_j= \nu_j \del\cdot[D_j(M, \vec{S})\del S_j + \bm{v}_j S_j] + f_j(M,\vec{S}),\label{eq:Sj}
\end{align} 
for $j=1,\dots,k$, where $M:Q\to\R$ denotes the biomass density and the vector-valued function $\vec{S}:Q\to \R^k$ the substrate concentrations. 
The biomass density $M$ is normalized with respect to the maximum biomass density and hence, it takes values in $[0,1)$.
The biomass diffusion coefficient $D_0:[0,1)\to [0,\infty)$ is degenerate, it satisfies $D_0(0)=0$ and $\lim_{m\nearrow 1} D_0(m)=\infty$. Although, we remark that large parts of our analysis are also valid for non-degenerate functions $D_0$.
The diffusion coefficients of the substrates $D_j : [0, 1] \times \R^k \to [0, \infty)$ are non-degenerate, i.e. they are bounded from above and below by positive constants. The constants $\nu_j\geq 0$ will be referred to as the mobility coefficients of the substrates. It is important to point out that the case of immobilized substrates ($\nu_j=0$) is included in our setting which leads to a coupling of Equation \eqref{eq:M} with ODEs in \eqref{eq:Sj}. Moreover, $\bm{v}_j:Q\to \R^d$ is a given flow-field. Finally, the reaction terms $f_0,\, f_j:\R^{k+1}\to \R$ describe the complex interplay between the substrates and biomass.

In biofilm modelling applications, it is important to allow for mixed Dirichlet-Neumann or homogeneous Neumann boundary conditions for $M$. To this end, we divide the boundary $\p\Om$ into two disjoint parts $\G_1$ and $\G_2$  that are both Lipschitz boundaries.
We complement \eqref{eq:M}--\eqref{eq:Sj} with the following initial and boundary conditions for $M$ and $\vec{S}$, 
\begin{align}\label{eq:boundary_initial_data}
M(0)&=M_0, \quad \vec{S}(0)=\vec{S}_0,\\
 M|_{\G_1}&=h_0, \quad [\del M\cdot\bm{\hat{n}}]|_{\G_2}=0,\quad \nu_j S_j|_{\p\Om}=\nu_jh_j,
\end{align}
\end{subequations}
where $\bm{\hat{n}}$ denotes the outward unit normal to $\p\Om$ and 
$M_0:\Om\to [0,1)$, $ \vec{S}_0:\Om\to \R^k$, $h_0:\Gamma_1\to [0,1)$ and $h_j:\p\Om\to \R$ are given. We remark that the case $\G_1=\emptyset$ is allowed in our setting which  corresponds to homogeneous Neumann boundary conditions for $M$. The case $\G_2=\emptyset$ is also included which corresponds to Dirichlet boundary conditions for $M$. Note that in \eqref{eq:boundary_initial_data} we do not prescribe boundary conditions for immobilized substrates $S_j$, i.e.  if $\nu_j=0$ for some $j\in\{1,\dots,k\}.$
To simplify the presentation of our results we assume Dirichlet boundary conditions for the substrates, but the analysis remains valid if we impose mixed boundary conditions for the substrates, see Remark \ref{remark:bcS}.  

In models for biofilm growth, the actual biofilm is described by the region where $M$ is positive, $$
\Omega^+(t)=\{x\in\Omega:M(t,x)>0\}.
$$ 
Due to the 
degeneracy of the biomass diffusion coefficient, $D_0(0)=0$, there is a sharp interface between the biofilm and the surrounding region, and the interface propagates at a finite speed. The additional singularity in the diffusion coefficient, $\lim_{m\nearrow 1} D_0(m)=\infty$, ensures that the biomass density does not exceed its maximum value, i.e. $M$ remains bounded by a constant strictly less than $1$. 

In \Cref{fig:biofilm} typical situations modelled by \eqref{eq:main} are sketched for biofilm colonies depending on a single substrate. In the left figure the substrate is dissolved in the spatial domain $\Omega$ and transported by diffusion and convection. The biofilm colony grows into the aqueous phase. In the right figure the substrate is immobilized and contained in the spatial domain $\Omega$. The bacteria consume and degrade the substrate, a biofilm front develops and propagates through the substratum.    

A system of the form \eqref{eq:main} with 
a single dissolved substrate $S=S_1$, i.e. $k=1$ and $\nu_1>0$, was first proposed in \cite{eberl2001new} to model biofilm growth in an aqueous medium. 
In this case,
\begin{subequations}\label{biofilm_model}
\begin{align}
D_0(M)&=d_2\frac{M^a}{(1-M)^b},& D_1(M,S)&=d_1,\\
f_0(M,S)&=k_3\frac{SM}{k_4+S}-k_2 M,& f_1(M,S)&= -k_1\frac{SM}{k_4+S},
\end{align}
\end{subequations}
for some constants $k_1,k_2,k_3,k_4,d_1,d_2>0$ and $a,b\geq 1$, and $\mathbf{v_1}$ is a given flow field.  
An ODE-PDE system of the form \eqref{eq:main} with a single substrate was used in \cite{eberl2017spatially} to model cellulolytic biofilms degrading an immobilized cellulose material. In this case, the functions $D_0,f_0$ and $f_1$ are as in \eqref{biofilm_model} and $D_1\equiv0$, $\mathbf{v_1}\equiv 0$.

\begin{figure}
    \begin{subfigure}[b]{0.49\textwidth}
         \centering
         \includegraphics[width=.99\textwidth]{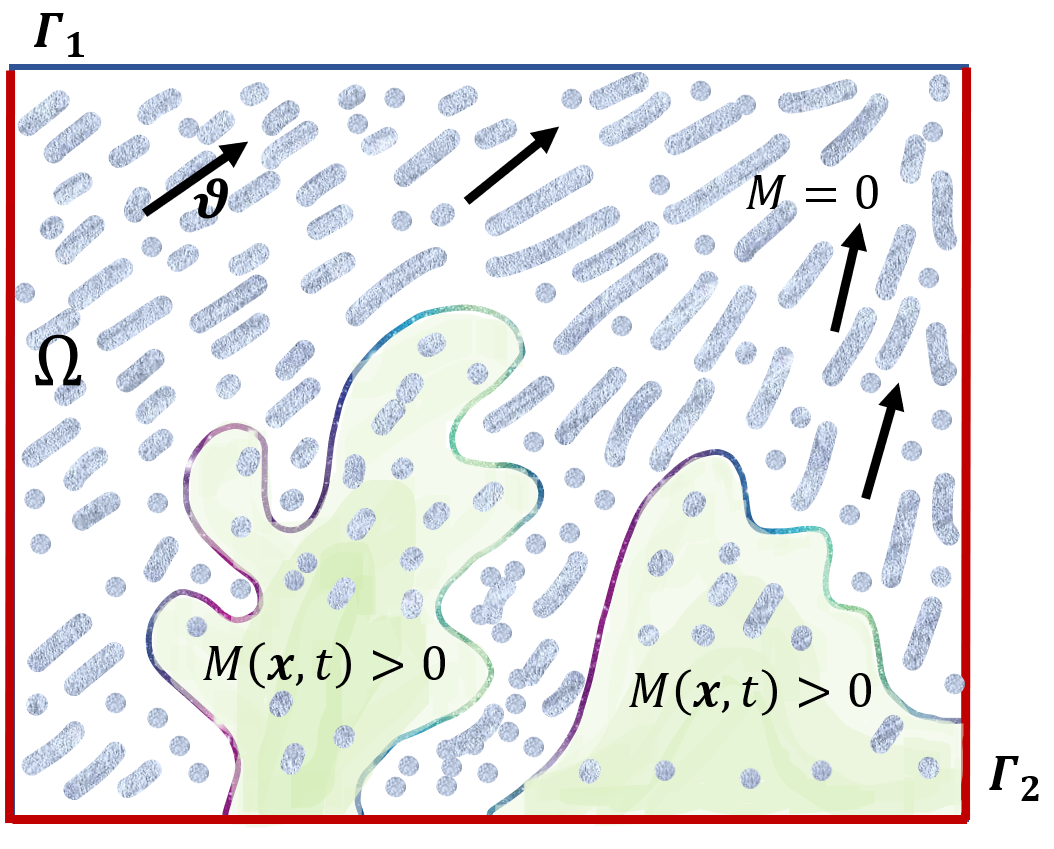}
         \caption{Coupled PDE-PDE systems}
         \label{fig:y equals x}
     \end{subfigure}
     \hfill
     \begin{subfigure}[b]{0.49\textwidth}
         \centering
          \includegraphics[width=.99\textwidth]{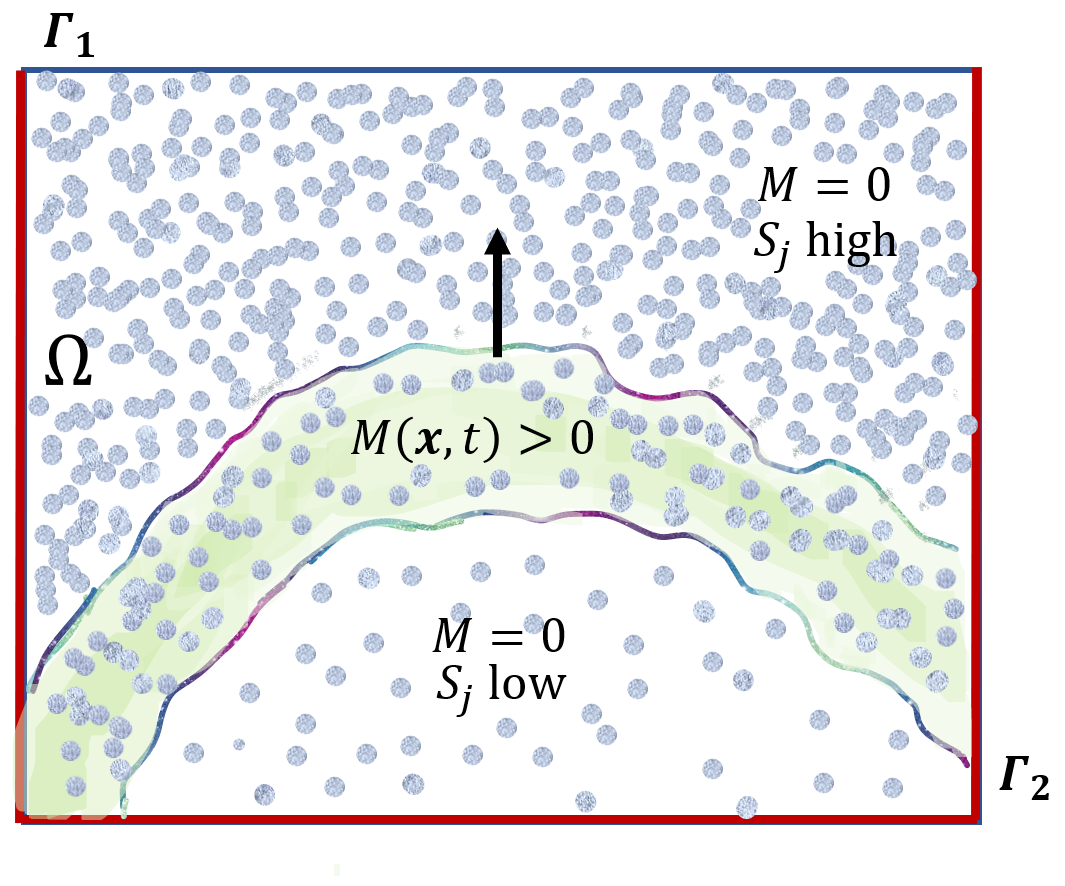}
         \caption{Coupled PDE-ODE systems}
         \label{fig:three sin x}
     \end{subfigure}
      \caption{
      Schematic figures illustrating biofilm growth in dependence of a single nutrient $S$ in an aqueous medium (a) and in an immobilized medium (b).
     The biofilm is represented by the region where $M(\bm{x},t)>0$, which is separated by the surrounding region by a sharp interface. Nutrients are consumed by  bacteria resulting in the production of biomass. The parts of the boundary where homogeneous Dirichlet and Neumann conditions are specified are also marked in the diagrams, $\Gamma_1$ (Dirichlet) in blue and $\Gamma_2$ (Neumann) in red.
    \\
    (a) PDE-PDE systems \cite{eberl2001new}: The biofilm colonies grow in a liquid containing substrates. The substrates diffuse and are transported by a flow field $\bm{v}$. The diffusion coefficient of the substrate might depend on $M$, i.e. it differs inside and outside the biofilm. 
    \\
    (b) PDE-ODE systems \cite{eberl2017spatially}: The bacteria degrade and consume an immobilized medium which is the case, e.g. for cellulolytic biofilms. The biofilm colony propagates consuming the immobile cellulose, leaving at its wake a region of low substrate concentrations.}
    \label{fig:biofilm}
\end{figure}
The existence of weak solutions of scalar nonlinear degenerate parabolic equations such as  \eqref{eq:M} was shown in the seminal papers \cite{alt1983quasilinear,alt1984nonstationary}, however, for bounded diffusion coefficients $D_0$. Uniqueness of solutions was proven in \cite{otto1996l1} using $L^1$-contraction. 
The existence of weak solutions for the biofilm model \cite{eberl2001new} with the diffusion coefficients and reaction functions in \eqref{biofilm_model} and $\mathbf{v_1}\equiv 0$ was proven in \cite{efendiev2009existence} under the assumptions of homogeneous Dirichlet boundary conditions for $M$, i.e. $\Gamma_2=\emptyset$ and $h_0\equiv0$.
The existence of the global attractor for the generated semigroup in $L^1(\Omega)$ was also shown. 
The well-posedness theory was generalized in  \cite{muller2022well} where more general functions $D_0$, $f_0$ and $f_1$ and mixed Dirichlet-Neumann boundary conditions were considered. The H\"older continuity of solutions was studied in \cite{victor2022}.  

Several extensions and variations of the single species biofilm growth model \cite{eberl2001new} have been proposed and analyzed. Most works are simulation studies and only few analytical results have been obtained. The well-posedness of multi-substrate biofilm models with $k>1$, $\nu_j>0$ in \eqref{biofilm_model}, appearing in antibiotic disinfection and quorum sensing applications, was established in \cite{sonner2015well,emerenini2017mathematical}. A PDE--ODE system with an immobile substrate, i.e. $k=1$ and $\nu_1=0$, was proposed and numerically studied in \cite{eberl2017spatially}. The simulations reproduced many experimentally observed features of  cellulolytic biofilms. The existence and stability of travelling wave solutions for this model were shown in \cite{mitra2022travelling}, but the well-posedness of the model remained an open problem. 
Many examples of semilinear coupled PDE-ODE models appearing in biology are discussed in \cite[Chapter 13]{murray2002mathematicalbiology}. 
For a PDE--ODE model for hysteretic flow through porous media with a diffusion coefficient $D_0$ depending on both $M$ and $\vec{S}$, the existence of solutions was shown in \cite{mitra2020existence}. 

We aim to develop a unifying solution theory for a large class of systems with degenerate diffusion that is motivated by models for biofilm growth, but the analysis is not limited to these applications. In fact, we expect that such models can also be used, e.g. to describe cancer cell invasion or the spread of wildfires.
In our paper we extend previous well-posedness results in the following directions: 

(a) \textit{Well-posedness results for PDE-PDE systems:} our results extend the theory developed for systems with one substrate in \cite{muller2022well} to systems with an arbitrary number of substrates $k\in \N$. Moreover, the existence of weak solutions is proven for a broad class of diffusion coefficients $D_0$ and $D_j$, reaction terms $f_j$, and allows for flow-fields $\bm{v}_j$ which has not been considered in earlier works. 

(b) \textit{Well-posedness of PDE--ODE systems ($\nu_j=0$):} 
The well-posedness of PDE-ODE systems of the form \eqref{eq:main} with a degenerate and/or singular diffusion coefficient $D_0$ has been an open problem. The theory we develop applies to the cellulolytic biofilm model \cite{eberl2017spatially} and implies its local well-posedness. 

(c) \textit{Mixed as well as homogeneous Neumann conditions for $M$:} Global well-posedness is shown for mixed Dirichlet-Neumann boundary conditions and a local well-posedness result is established assuming homogeneous Neumann boundary conditions for $M$.  
Moreover, apart from well-posedness results we also analyze qualitative properties such as boundedness or blow-up of solutions. 

(d)  \textit{Global spatial regularity of M:} We further show that under certain porous medium type growth conditions on $D_0$ close to zero, the biomass concentration $M$ inherits some  global spatial regularity. 

The outline of our paper is as follows: 
In Section \ref{sec:math} we introduce notation, state our assumptions on the data and introduce the concept of weak solutions. In Section \ref{sec:mixed} we prove global well-posedness for systems with Dirichlet or mixed Dirichlet-Neumann boundary conditions for $M$. In \Cref{sec:Neumann} we establish local well-posedness for systems with homogeneous Neumann conditions for $M$. We also derive 
criteria ensuring finite-time blow-up of the model and discuss some important examples.
In \Cref{sec:fronts} we show that even in the degenerate case, the biomass density $M$ possesses some global spatial regularity.

\section{Problem formulation}\label{sec:math}

In this section, we introduce notation and a suitable functional framework. We state the properties of the coefficient functions and the boundary and initial data for system \eqref{eq:main} that will be assumed throughout the paper. Moreover, we introduce weak solutions of the problem.

\subsection{Preliminaries}\label{sec:prelim}

\paragraph{Functional setting:} 

Let $\Om\subset\R^d$ be a bounded Lipschitz domain. The boundary $\p\Om$ is divided into two regular open  subsets $\G_1$ and $\G_2$  that are both Lipschitz boundaries and such that $\partial \Omega=\overline{\G}_1\cup \overline{\G}_2$ and $\G_1\cap\G_1=\emptyset$, e.g. see \cite{salsa2016}.
We denote by $(\cdot,\cdot)$ and $\|\cdot\|$ the $L^2(\Om)$ inner product and norm. The norm of any other Banach space $V$ will be denoted by $\|\cdot\|_V$.
For $1\leq p\leq \infty$, let $W^{1,p}(\Om)$ denote the Sobolev space of functions $u\in L^p(\Om)$ such that the weak derivative $\del u$ exists and $\del u\in (L^p(\Om))^d$. For $r\in (0,1)$ and $p\in [0,\infty)$, the Sobolev-Slobodeckij space $W^{r,p}(\Om)$ is the set of functions $u\in L^p(\Om)$ such that
\begin{align}
\|u\|_{W^{r,p}(\Om)}:=\|u\|_{L^p(\Om)} + \int_{\Om}\int_{\Om}\dfrac{|u(\bm{x})-u(\bm{y})|^p}{|\bm{x}-\bm{y}|^{d+ rp}}d\bm{x} d\bm{y}<\infty.
\end{align}
We define $H^r(\Om):=W^{r,2}(\Om)$ for $r\in (0,1]$. Let $H^1_0(\Om)$ denote the  closure of $C^\infty_c(\Om)$ in $H^1(\Om)$, which is equipped with the norm $\|u\|_{H^1_0(\Om)}:=\|\del u\|$. Similarly, we define
\begin{subequations}\label{eq:SobolevSpaces}
\begin{align}
\H:=\{u\in H^1(\Om):\; \mathrm{tr}(u)=0 \text{ in } \G_1 \}\quad \text{with the norm} \quad \|u\|_{\H}:=\|u\|_{H^1(\Om)}.
\end{align}
The dual spaces of $H^1_0(\Om)$ and $\H$ are defined as:
\begin{align}
H^{-1}:=(H^1_0(\Om))^*\quad \text{ and } \quad \Hm=(\H)^*.
\end{align}
Observe that,
\begin{align}\label{eq:HisH1}
\text{ if } \Gamma_1=\emptyset \text{ then } \H=H^1(\Om),\quad  \text{ if }  \Gamma_2=\emptyset \text{ then } \H=H^1_0(\Om). 
\end{align}
\end{subequations}
Let $\langle \cdot,\cdot \rangle$ denote the duality pairing of $\H$ and $\Hm$. The duality pairing of any other Sobolev space $V$ will be denoted by $\langle \cdot,\cdot\rangle_{V,V^*}$.

Finally, we introduce the following Bochner spaces that are important for our analysis:
\begin{subequations}\label{eq:Bochner}
\begin{align}
&\W:=L^\infty(0,T;L^\infty(\Om)) \cap H^1(0,T;\Hm)\cap C([0,T];L^2(\Om)),\\
&\X:=L^2(0,T;\H)\cap H^1(0,T;\Hm),\\
&\Y:=L^\infty(0,T;H^1(\Om))\cap H^1(0,T;L^2(\Om)),\\
&\Z:=C([0,T];(L^2(\Om))^k).
\end{align}
\end{subequations}
Note that we have the continuous embedding $\X\hookrightarrow C([0,T];L^2(\Omega))$.

\paragraph{Inequalities:} 
Note that the Poincar$\acute{\text{e}}$ inequality, i.e. $\|u\|\leq C_\Om\|\del u\|$ for $u\in H^1_0(\Om)$, where $C_\Om>0$ denotes the Poincar\'e constant, 
also holds for functions $u\in \H$ if $\Gamma_1\not= \emptyset$. 

We recall Young's inequality 
stating that for any $\sigma>0$ one has
\begin{equation}\label{Eq:YoungsIneq}
ab\leq \frac{1}{2\sigma}a^2+ \frac{\sigma}{2}b^2  \qquad \forall a,b\in\R. 
\end{equation}

We will also frequently use Gronwall's Lemma stating that if $u,\,a,\,b\in C(\R)$ are non-negative, then $u(t) \leq a(t) +  \int_0^t u(\vr)\,b(\vr)\,\dd\vr$   implies that
\begin{subequations}
\begin{equation}\label{eq:Gronwall}
u(t) \leq a(t) +  \int_0^t a(\vr)\,b(\vr)e^{\int_\vr^t b(\tau)d\tau }\, \dd\vr
\end{equation}
for all $t>0$; and the discrete counterpart of the Gronwall Lemma: Let $\{u_n\}_{n\in \N}$, $\{a_n\}_{n\in \N}$, $\{b_n\}_{n\in \N}$ be non-negative sequences such that $u_n\leq a_n + \sum_{k=1}^{n-1} b_k u_k$. Then 
\begin{equation}\label{eq:discGronwall}
u_n\leq a_n + \sum_{k=1}^{n-1} a_kb_k \exp\left( \sum_{k<j<n} b_j\right).
\end{equation}
\end{subequations}
Finally, for a convex $\eta\in C(\R^+)$ with $\eta(0)=0$ we will use Jensen's inequality and the super-additivity property:
\begin{subequations}\label{eq:ConvexFuncIneq}
\begin{align}
&\text{Jensen's inequality: }&&\eta\left (\tfrac{1}{|\Om|}\int_{\Om} |f| \right )\leq \tfrac{1}{|\Om|} \int_{\Om} \eta(|f|)\qquad \text{ for } f\in L^1(\Om);\label{eq:ConvexFuncIneq1}\\
&\text{Super-additivity: }&& \eta(a)+ \eta(b)\leq \eta(a+b)\qquad \text{ for all } a,\,b\geq 0.\label{eq:ConvexFuncIneq2}
\end{align}
\end{subequations}

\paragraph{Further notation:} We denote by $[\cdot]_+$ and $[\cdot]_-$ the positive and negative part of functions, i.e. $[\cdot]_+:=\max\{\cdot,0\}$ and $[\cdot]_-:=\min\{\cdot,0\}$, respectively. 
By $C>0$  we refer to an undisclosed constant in the estimates that may vary in each occurrence and from line to line. Finally, 
the notation 
\begin{align}\label{eq:lesssim}
a\lesssim b \quad \text{ implies that } a\leq C b \quad \text{for some constant } C>0
\end{align}
which does not depend on a parameter $\e>0$ (to be specified later).

\subsection{Assumptions on the data}\label{sec:prop_func}

We specify the hypotheses on the data associated with \eqref{eq:main}. 
\begin{enumerate}[label=(P\theEUXproperties)]
 \setlength\itemsep{-0.2em}
\item The diffusion coefficient $D_0:[0,1)\mapsto [0,\infty)$ is a continuous function that is strictly increasing in  $[0,\epsilon_0)$ for some $\epsilon_0\in (0,1]$, and satisfies
$$D_0(0)=0,\;\; \lim\limits_{m\nearrow 1}D_0(m)=\infty\;\text{ and } \;\;D_0(m)>0 \text{ for all } m\in (0,1).$$
The primitive of $D_0$, expressed by the Kirchhoff transform function $\Phi:[0,1)\to [0,\infty)$, $\Phi(m)=\int_0^m D_0(\vr)\,\dd\vr,$ satisfies
\begin{align}
\lim\limits_{m\nearrow 1}\Phi(m)=\infty.
\label{eq:defPhi}
\end{align}
\label{prop:D}
\stepcounter{EUXproperties}

\item The diffusion coefficients $D_j:[0,1]\times \R^k\to [D_{\min},D_{\max}]$, $j=1,\dots,k,$ with constants $0<D_{\min}<D_{\max}<\infty$, are Lipschitz continuous with respect to both variables.\label{prop:Df}
\stepcounter{EUXproperties}
\end{enumerate}

\begin{remark}[Biofilm models]
In models for biofilm growth, see e.g. \cite{eberl2017spatially,efendiev2009existence}, the diffusion coefficient $D_0$ is given by the function in \eqref{biofilm_model}, and the diffusion coefficient $D_1$ for the (single) substrate is assumed to be constant.
These functions satisfy all assumptions in \ref{prop:D}--\ref{prop:Df}. 
\end{remark}

\begin{remark}[Generalizations of the assumptions \ref{prop:D}--\ref{prop:Df}]
Our analysis can be extended to systems where the diffusion coefficient $D_0$ is piecewise constant, non-degenerate, and/or has a porous media type degeneracy, e.g., $D_0(m)=m^a$, for some constant $a>1$. To keep the analysis uniform and self-contained, we only analyze the case \ref{prop:D} which is more involved and arises in models for biofilm growth.

 
We could also allow for degenerate diffusion coefficients $D_j$ if they only depend on the substrate $S_j$. Some additional assumptions are required to cover this case which are discussed in \Cref{cor:degenerateDj}.
\end{remark}

For the flow-field and reaction terms we make the following assumptions:
\begin{enumerate}[label=(P\theEUXproperties)]
 \setlength\itemsep{-0.2em}
\item The flow-field satisfies $\bm{v}_j\in (L^\infty(Q))^d$, $j=1,\dots,k$.\label{prop:fl}
\stepcounter{EUXproperties}
\item 
The functions $f_0,\,f_j\in C([0,1]\times \R^{k})$ are uniformly Lipschitz continuous. 
They can be extended to uniformly Lipschitz continuous functions on $\R^{k+1}$ which (to simplify notation) we will also denote by $f_0,\,f_j$. The constant  $C_L\geq 0$ is the maximum of the Lipschitz constants of $f_0,f_1,\dots, f_k$.
Moreover,  $f_0(0,\vec{s})\geq 0$ for all $\vec{s}\in \R^k$. \label{prop:fg1}
\end{enumerate}
\vspace{-1.5em}

\begin{enumerate}[label=(P\theEUXproperties*)]
\item There exists a non-negative and locally Lipschitz continuous function $f_{max}\in C(\R)$ such that $f_0(\cdot,\vec{s}) \leq f_{\max}(\cdot)$ for all $\vec{s}\in \R^k$. \label{prop:fg2}\stepcounter{EUXproperties}
\end{enumerate}

\begin{remark}[Assumptions \ref{prop:fg1}--\ref{prop:fg2}]
Assumption \ref{prop:fg1} admits reaction functions $f_0(\cdot,\vec{s})$, $f_j(\cdot,\vec{s})$ (for $\vec{s}\in \R^k$) that have superlinear growth with respect to their first argument as long as they are Lipschitz continuous within the interval $[0,1]$. This is because the physically relevant solutions satisfy $M\in [0,1]$. However, before proving the upper bound for $M$ (in \Cref{lemma:MaxRegSol}), we need the functions $f_0$ and $f_j$ to be defined in $\R^{k+1}$ (in \Cref{lemma:ExistRegSol}) which is why we introduce the extensions.

Assumption \ref{prop:fg2} is needed to derive the $L^\infty$ bound for the solution $M$. This is important in our setting since physically relevant solutions take values in $[0,1)$, otherwise, the models are not valid. Such $L^\infty$ bounds may not be required in other applications, for example for porous medium-type equations. We therefore explicitly state in all theorems and lemmas where this assumption is required and where it can be omitted.

Under additional assumptions, the analysis can be generalized also to
systems with reaction functions that depend on $\bm{x}$ and $t$, e.g. see \cite{muller2022well}. To simplify the presentation of our results, we omit this dependency here.  

Finally, we remark that the condition $f_0(\cdot,\vec{s})\leq f_{\max}(\cdot)$ can be relaxed to $f_0(\cdot,\vec{s})\leq g(|\vec{s}|)\,f_{\max}(\cdot)$ for some $g\in C(\R^+)$.  Then, for the proofs to go through, we need uniform $L^\infty$ bounds on the solution $(M,\vec{S})$ which can be established for a certain class of functions $f_j$. For an example, we refer to \Cref{cor:degenerateDj}. 
\end{remark}

For the boundary and initial data we assume the following properties: 
\begin{enumerate}[label=(P\theEUXproperties)]
 \setlength\itemsep{-0.2em}
\item The initial data $M_0\in L^\infty(\Om)$ satisfies  
$$
\underline{M}:=\mathrm{ess}\inf\limits_{\bm{x}\in \Om} \{M_0\}\geq 0\quad \text{and}\quad
\overline{M}:=\mathrm{ess}\sup\limits_{\bm{x}\in \Om} \{M_0\}<1.
$$ 
The initial data $\vec{S}_0\in (L^\infty(\Om))^k$ satisfies 
$$
\underline{S}:=\min\limits_{1\leq j\leq k}\mathrm{ess}\inf\limits_{\bm{x}\in \Om}\{S_{0,j}\}>-\infty\quad \text{and} \quad \overline{S}:=\max\limits_{1\leq j\leq k}\mathrm{ess}\sup\limits_{\bm{x}\in \Om}\{S_{0,j}\}<\infty.
$$
\label{prop:IC}
\stepcounter{EUXproperties}
\item The Dirichlet boundary data $h_0:\G_1\to [\underline{M},\overline{M}]$ is such that there exists $h^e_0\in H^1(\Om)$ satisfying $h^e_0|_{\G_1}=h_0$ in a trace sense. If $\Gamma_1=\emptyset$ then set $h^e_0\equiv 0$. For the Dirichlet data $h_j:\p\Om\to [\underline{S},\overline{S}]$ there  also exist functions $h^e_j\in H^1(\Om)$ such that $h^e_j|_{\partial\Om}=h_j$ in a trace sense.
\label{prop:BC} \stepcounter{EUXproperties}
\end{enumerate}

\begin{remark}[Assumption \ref{prop:BC}]\label{remark:bcS}
Observe that, under the assumptions in \ref{prop:BC}, it is always possible to choose the extensions to $\Om$ such that  $h^e_0\in [\underline{M},\overline{M}]$ and $h^e_j\in [\underline{S},\overline{S}]$ a.e. in $\Om$. For example, consider $\bar h^e=\min\{h^e_0,\overline{M}\}\in H^1(\Om)$. Then $\bar h^e\leq \overline{M}$ a.e. and $\bar h^e=h_0$ on $\partial\Om$. 
This choice will implicitly be used in the proofs that follow. Similar arguments apply to the boundary conditions $h_j.$

To keep notations simple we only consider Dirichlet boundary conditions for the substrates $\vec{S}$. Mixed Dirichlet-/Neumann boundary conditions with different divisions of the boundary depending on $j$ can also be assumed without major modifications in the subsequent arguments, see e.g. \cite{muller2022well}.
\end{remark}

\subsection{Weak solutions}
We introduce the following notion of weak solutions.
\begin{definition}[Weak solution] The pair $(M,\vec{S})$ with $M\in \W$ (see \eqref{eq:Bochner}), $\Phi(M)\in L^2(0,T;H^1(\Om))$, $S_j\in H^1(0,T;H^{-1}(\Om))\cap C([0,T];L^2(\Om))$, and $\nu_j S_j\in L^2(0,T;H^1(\Om))$, $j=1,\dots,k,$  is a weak solution of \eqref{eq:main} provided that $M$ is bounded in $[0,1)$ a.e. in $Q$, $M(0)=M_0$ and  $\vec{S}(0)=\vec{S}_0$ a.e. in $\Omega$, $\Phi(M)=\Phi(h_0)$ on $\G_1$ and $\nu_j S_j=\nu_j h_j$ on $\p\Om$ in the trace sense, and for all $\f\in L^2(0,T;\H)$, $\vec{\z}\in L^2(0,T;(H^1_0(\Om))^k)$, we have
\begin{subequations}\label{eq:weak}
\begin{align}
&\int_0^T \langle \f,\p_t M \rangle + \int_0^T (\del \Phi(M),\del \f)= \int_0^T (f_0(M,\vec{S}),\f),\\
&\int_0^T \langle \z_j,\p_t S_j \rangle_{H^1_0,H^{-1}} + \nu_j \int_0^T ( D_j(M,\vec{S})\del S_j + \bm{v}_j\, S_j,\del \z_j)= \int_0^T (f_j(M,\vec{S}),\z_j).
\end{align}
\end{subequations}\label{def:WeakSol}
\end{definition}

\begin{remark}
In \Cref{def:WeakSol} we take $\nu_j \,S_j\in  L^2(0,T;H^1(\Om))$ instead of taking $S_j\in L^2(0,T;H^1(\Om))$. This is required since $\nu_j=0$ is allowed in our setting, and in this case, $S_j$ might not possess any spatial regularity. Similarly, we see that $\Phi(M)$, and not $M$, possesses spatial regularity. Therefore, the traces are also only defined for functions with sufficient spatial regularity.
\end{remark}

\section{Well-posedness for Dirichlet and mixed boundary conditions}\label{sec:mixed}

In this section, we prove the well-posedness of weak solutions for the case when $\G_1$ has non-zero measure, i.e. $M$ either satisfies Dirichlet boundary condition or mixed Dirichlet-Neumann boundary condition.
The main results of this section are stated in the following theorems. 

\begin{theorem}[Existence and boundedness]\label{theo:Existence}
Let \ref{prop:D}--\ref{prop:BC} and \ref{prop:fg2} be satisfied, and
$\G_1$ have  non-zero measure. Then, there exists a weak solution $(M,\vec{S})$ of\eqref{eq:main}  in the sense of \Cref{def:WeakSol}. Furthermore, a constant $\d\in (0,1)$ exists such that  $0\leq M\leq  1-\d$ a.e. in $Q$.
\end{theorem}

\begin{theorem}[Uniqueness]\label{theo:ExistUnique}
Let the assumptions of Theorem \ref{theo:Existence} hold. In addition, for each $j\in \{1,\dots,k\}$, assume that either $\nu_j=0$ or the diffusion coefficient $D_j$ depends only on $S_j$.  Then, a unique  weak solution $(M,\vec{S})$ of\eqref{eq:main} exists  in the sense of \Cref{def:WeakSol}.
\end{theorem}

The proof of Theorem \ref{theo:ExistUnique} is based on a contraction argument which, along with the existence of solutions, guarantees uniqueness. However, a different argument based on Schauder's fixed  point theorem is required for proving the existence of solutions in the more general setting of Theorem \ref{theo:Existence}.

For the proof of these theorems, we initially focus on the first equation in \eqref{eq:weak}  for a given $\vec{S}$. In \Cref{sec:RegularizedProblem} we consider a non-degenerate approximation of \eqref{eq:M} and then discuss existence (\Cref{lemma:ExistRegSol}) and boundedness (\Cref{lemma:MaxRegSol}) of solutions. In \Cref{sec:Unreg}, we pass the regularization parameter to zero to show the existence of weak solutions of the original problem (\Cref{lemma:ExistUnRegSol}). We then consider the coupled system. In \Cref{sec:L1contraction}, the $L^1$ contraction principle (\Cref{lemma:L1cont}) is applied to prove \Cref{theo:ExistUnique}, and in \Cref{sec:Schauder}, a Schauder argument (\Cref{lemma:Schauder}) is used to prove \Cref{theo:Existence}.

We remark that Lemmas \ref{lemma:ExistRegSol}--\ref{lemma:Schauder} hold for all boundary conditions including homogeneous Neumann boundary condition. 
The proof of \Cref{lemma:ExistRegSol} is postponed to Section \ref{sec:Neumann}, due to complications arising from homogeneous Neumann condition.
Lemmas \ref{lemma:MaxRegSol}--\ref{lemma:Schauder} are proven in the general case.

\subsection{A regularized problem}\label{sec:RegularizedProblem}

We introduce the following regularization of the  Kirchhoff transform $\Phi$: for $\e>0$, let $\Phi_\e\in C^1(\R)$  be a non-degenerate approximation of $\Phi$ satisfying
\begin{align}\label{eq:propPhieps}
\e\leq {\Phi_\e}'\leq \e^{-1}\quad\text{ and }\quad  \lim\limits_{\e\to 0}\Phi_\e(m)= \Phi(m) \quad \text{ for all } m\in [0,1).
\end{align}
A specific choice of $\Phi_\eps$ that will be used in the sequel and in \Cref{sec:fronts} is
\begin{align}\label{eq:PhiepsDef}
\Phi_\e(m):=\int_0^m \min\left\{\max\{\e,D(\vr)\},\e^{-1}\right\}\, \dd \vr.
\end{align}
Then, recalling the functional spaces defined in \eqref{eq:Bochner},  the following lemma holds.

\begin{lemma}[Existence for a regularized problem]\label{lemma:ExistRegSol}
Let \ref{prop:D}--\ref{prop:BC} hold. 
Let $\vec{s}\in \Z$ be given and $\varepsilon\in(0,1)$ be sufficiently small. 
Then there exists a unique $M_{s,\e}\in \X+h^e_0$ which satisfies $M_{s,\e}(0)=M_0$, and for all  $\f\in L^2(0,T;\H)$,
\begin{align}\label{eq:reg}
\int_0^T \langle \f,\p_t M_{s,\e} \rangle + \int_0^T (\del \Phi_\e(M_{s,\e}),\del \f)= \int_0^T (f_0(M_{s,\e},\vec{s}),\f).
\end{align}
Moreover, $M_{s,\e}\in C([0,T];L^2(\Omega))$, and for all $t\in [0,T]$ we have
\begin{align}
&\|M_{s,\e}(t)\|^2 + \int_0^T \left[\|\del \Phi_\e(M_{s,\e})\|^2 + \|\p_t M_{s,\e}\|^2_{\Hm}\right]\nonumber\\
&\lesssim 1+\int_0^T \left(\|\vec{s}\|^2+\|\Phi_\e(M_{s,\e})\|^2\right)+\left(1+\|\Phi'_\e(M_{s,\e})\|_{L^\infty(Q)}\right)\| h^e_0\|^2_{H^1(\Om)}.\label{eq:apriori1}
\end{align}
Furthermore, if $M_0\in H^1(\Om)$ in \ref{prop:IC}, and $M_0|_{\G_1}=h_0$ in the trace sense then, in addition, it holds that
\begin{align}\label{eq:apriori2}
&\|\del \Phi_\e(M_{s,\e}(t))\|^2+ \int_0^T\int_{\Om} \frac{|\p_t \Phi(M_{s,\e})|^2}{\Phi'_\e(M_{s,\e})}\nonumber\\
&\lesssim \|\del \Phi_\e(M_0)\|^2 + \|\Phi'_\e(M_{s,\e})\|_{L^\infty(Q)}\left [1+ \int_0^T \|\vec{s}\|^2 \right ].
\end{align}
\end{lemma}

\noindent  The proof of \Cref{lemma:ExistRegSol} is postponed to \Cref{sec:ExistNeumann}.
For $\G_1$ having non-zero measure, the existence of $M_{s,\e}\in \X+h^e_0$ follows immediately from \cite{alt1983quasilinear} since ${\Phi_\e}'$ satisfies the uniform ellipticity condition. 
However, the result in \cite{alt1983quasilinear} does not cover the case of homogeneous Neumann conditions, i.e. the case $\G_1=\emptyset$. 

The assumption \ref{prop:fg2} is not required in \Cref{lemma:ExistRegSol}, but it is needed for the next result.

\begin{lemma}[Boundedness for the regularized problem]\label{lemma:MaxRegSol}
In addition to the hypothesis of Lemma \ref{lemma:ExistRegSol} we assume that \ref{prop:fg2} holds. 
Let $M_{s,\e}\in \X +h^e_0$ denote the solution in \Cref{lemma:ExistRegSol}. 
Moreover, let $\hat{M}\in C^1(\R^+)$ denote the solution of the integral equation
\[
\hat{M}(t)=\overline{M} + \int_0^t f_{\max}(\hat{M}(\vr))\,\dd\vr, \qquad t\in[0,T].
\]
\begin{enumerate}[label=(\alph*)]
\item Then, $0\leq M_{s,\e}(t)\leq \hat{M}(t)$ a.e. in $\Om$ for all $t\in [0,T]$.
\item If, in addition, $\G_1$ has non-zero measure, then a constant $\d\in (0,1)$ exists such that
\[
0\leq M_{s,\e}(t) \leq 1-\d \ \text{ a.e. in } \Om \text{ for all } t\in [0,T].
\]
\end{enumerate}
\end{lemma}

Observe that the above lemma implies that
$ M_{s,\e}\in L^\infty(0,T;L^\infty(\Om))$ 
and the family $ M_{s,\e}$ is uniformly bounded with respect to  $\e>0$ in $L^\infty(0,T;L^\infty(\Om))$.

\begin{proof} The existence of $\hat{M}$ follows from the Picard-Lindel\"of Theorem since $f_{\max}\in \mathrm{Lip}(\R)$. Moreover, it satisfies $\p_t \hat{M}= f_{\max}(\hat{M})$ and therefore, $\hat{M}\geq \overline{M}$ since $f_{\max}$ was assumed to be non-negative in \ref{prop:fg2}.

\textbf{(Step 1) $\mathbf{M_{s,\e}\geq 0}$:} Inserting the test function $\f=[M_{s,\e}]_-$ in \eqref{eq:reg} implies that
$$
\int_0^T \left [ \p_t\left ( \tfrac{1}{2}\|[M_{s,\e}]_-\|^2\right ) + \e\|\del [M_{s,\e}]_-\|^2 \right ]\overset{\ref{prop:fg1}}\leq C_L\int_0^T \|[M_{s,\e}]_-\|^2.
$$
Since $[M_{s,\e}(0)]_-=[M_0]_-=0$, we have $\|[M_{s,\e}]_-(T)\|=0$ using Gronwall's Lemma \eqref{eq:Gronwall}. 

\vspace{.5em}
 \textbf{(Step 2) $\mathbf{M_{s,\e}\leq \hat{M}}$:} Inserting the test function $\f=[M_{s,\e}-\hat{M}]_+\in L^2(0,T;\H)$ in \eqref{eq:reg} we obtain 
\begin{subequations}\label{eq:MaxPrinIneq}
\begin{align}
&\int_0^T \langle [M_{s,\e}-\hat{M}]_+, \p_t M_{s,\e} \rangle =\int_0^T \langle [M_{s,\e}-\hat{M}]_+,\p_t [M_{s,\e}-\hat{M}] \rangle + \int_0^T \langle [M_{s,\e}-\hat{M}]_+,\p_t \hat{M} \rangle\nonumber\\
& \qquad\, =\int_0^T  \p_t \left (\frac{1}{2}\|[M_{s,\e}-\hat{M}]_+\|^2 \right ) + \int_0^T ( \p_t \hat{M},[M_{s,\e}-\hat{M}]_+ )\nonumber\\
& \qquad \overset{\ref{prop:IC}}= \frac{1}{2}\|[M_{s,\e}-\hat{M}]_+(T)\|^2  + \int_0^T ( \p_t \hat{M},[M_{s,\e}-\hat{M}]_+ ),\\
&\int_0^T (\del \Phi_\e(M_{s,\e}),\del [M_{s,\e}-\hat{M}]_+) =  \int_0^T ( {\Phi_\e}'(M_{s,\e})\del M_{s,\e},\del [M_{s,\e}-\hat{M}]_+)\geq 0,
\end{align}
and for the reaction term we obtain 
\begin{align}
&\int_0^T (f_0(M_{s,\e},\vec{s}),[M_{s,\e}-\hat{M}]_+)=\int_0^T (f_0(M_{s,\e},\vec{s}) - f_0(\hat{M},\vec{s}) + f_0(\hat{M},\vec{s}),[M_{s,\e}-\hat{M}]_+)\nonumber\\
\qquad &\overset{\ref{prop:fg1},\ref{prop:fg2}}\leq C_L \int_0^T \|[M_{s,\e}-\hat{M}]_+\|^2 + \int_0^T f_{\max}(\hat{M}) [M_{s,\e}-\hat{M}]_+.
\end{align}
\end{subequations}
Combining the estimates in \eqref{eq:MaxPrinIneq} it follows that
\begin{align}
\frac{1}{2}\|[M_{s,\e}-\hat{M}]_+(T)\|^2  + \int_0^T ( \p_t \hat{M} -f_{\max}(\hat{M}),[M_{s,\e}-\hat{M}]_+ )\leq C_L \int_0^T \|[M_{s,\e}-\hat{M}]_+\|^2.
\end{align}
The second term is zero by the definition of $\hat{M}$. Hence, using Gronwall's Lemma \eqref{eq:Gronwall} we have the result.

 \textbf{(Step 3) $\mathbf{M_{s,\e}\leq 1-\d}$:} This is a generalization of Proposition 6 in \cite{efendiev2009existence} to the case of mixed or homogeneous Neumann boundary conditions, see also the proof of Theorem 2.7 in \cite{muller2022well}. 
 For $f_{\max}(\cdot)$ introduced in \ref{prop:fg2}, let $\hat{u} \in \H$  solve the elliptic problem
 \begin{align}
 (\del \hat{u},\del \f) =( \hat{C},\f)\quad \text{for all}\ \f\in \H,\quad 
  \text{where } \hat{C}:=\max_{0\leq t\leq T}f_{\max}(\hat{M}(t)).
 \end{align}
 The existence of a unique weak solution $\hat{u}$ directly follows from the Lax-Milgram Lemma. If $d=1$ (one space-dimension), then we immediately have  $\hat{u}\in L^\infty(\Om)$ from Morrey's inequality \cite[Chapter 5]{evans1988partial}. Hence, let $d\geq 2$. Set $q=2d/(d-2)$ for $d>2$, and $q>2$ for $d=2$. Then for $m\geq 0$, inserting the test function $\f=[\hat{u}-m]_+\in \H$ and denoting $A(m):=\{\bm{x}\in \Om: \hat{u}(\bm{x})>m\}$  we have the estimates,
\begin{align*}
\|\del [\hat{u}-m]_+\|^2&\leq \hat{C} \,\|[\hat{u}-m]_+\|_{L^1(\Om)},\\
 \|[\hat{u}-m]_+\|_{L^1(\Om)}&\leq |A(m)|^{1-\frac{1}{q}}\|[\hat{u}-m]_+\|_{L^q(\Om)}\leq C' |A(m)|^{1-\frac{1}{q}}\|\del [\hat{u}-m]_+\|,
\end{align*}
where the last inequality follows from the Sobolev inequality \cite[Chapter 5]{evans1988partial}. Hence, we have $\|[\hat{u}-m]_+\|_{L^1(\Om)}\leq {C'}^2 \hat{C} |A(m)|^\g$, where $\g=2-\frac{2}{q}>1$.
Thus, following the steps of \cite[Lemma 7.3]{hartman1966some} we conclude that $\hat{u}\in L^\infty(\Om)$. 
Hence, using the comparison principle \cite{otto1996l1}, one has $\Phi_\e(M_{s,\e}(t))\leq \hat{u} + \bar{M}<\infty$ a.e. in $\Om$ for all $\e>0$ and $t>0$, which concludes the proof.
\end{proof}

\begin{remark}[Generalization to $\vec{s}\in (L^2(Q))^k$]\label{rem:SinL2Q}
Although \Cref{lemma:ExistRegSol,lemma:MaxRegSol}, and the following \Cref{lemma:ExistUnRegSol}, assume $\vec{s}\in \Z$ to simplify the presentation, the results remain valid for all $\vec{s}\in (L^2(Q))^k$ as evident from the a-priori estimates \eqref{eq:apriori1}--\eqref{eq:apriori2}. This observation will become important in \Cref{lemma:Schauder} which provides the setting for the proof of  \Cref{theo:Existence}.
\end{remark}

\subsection{Existence of solutions of the degenerate parabolic problem}\label{sec:Unreg}

\begin{lemma}[Existence for the degenerate problem]\label{lemma:ExistUnRegSol}
Let \ref{prop:D}--\ref{prop:BC} and \ref{prop:fg2} hold. 
Let $\vec{s}\in\Z$ be given and let $0<T^*\leq\infty$ denote the time, independent of $\vec{s}$ and $\e>0$, such that the solutions $M_{s,\e}\in \X+h^e_0$ of \eqref{eq:reg} remain bounded in $[0,1)$ for all $t< T^*$. Let $T<T^*$. Then there exists a unique $M_s\in \W$ with $\Phi(M_s)\in L^2(0,T;H^1(\Om))$ satisfying $M_s(0)=M_0$, $\Phi(M_s)=\Phi(h_0)$ on $\G_1$ in the trace sense, and 
\begin{align}\label{eq:unreg}
\int_0^T \langle \f,\p_t M_s \rangle + \int_0^T (\del \Phi(M_s),\del \f)= \int_0^T (f_0(M_s,\vec{s}),\f),
\end{align}
for all $\f\in L^2(0,T;\H)$. Moreover, $0\leq M_s<1-\d$ a.e. in $Q$ for some constant $\d\in (0,1)$.
\end{lemma}

\begin{proof} In this proof, we will first assume that
\begin{align}
 M_0\in H^1(\Om).\label{eq:M0inH1}
\end{align}
This constraint will later be dropped. \Cref{lemma:MaxRegSol} and the assumption $T<T^*$, imply the existence of $\bar{\d}\in (0,1)$ such that 
\begin{align}\label{eq:MPhiLess1}
M_{s,\e}\in [0,1-\bar{\d}],\ \  \text{and consequently, } \ \ \phi_\e:=\Phi_{\e}(M_{s,\e})\in [0,\Phi(1-\bar{\d})] \  \text{ a.e. in } Q,
\end{align}
for small $\e>0$. The shorthand $\phi_\e$ will be used to denote $\Phi_{\e}(M_{s,\e})$ for the rest of the proof.

\textbf{(Step 1) Convergence of $\mathbf{M_{s,\e}}$ and $\mathbf{\phi_\e}$, assuming \eqref{eq:M0inH1}:} Taking \eqref{eq:MPhiLess1} into account which implies that ${\Phi_\e}'(M_{s,\e})$ is bounded above independent of $\e$, we conclude from \eqref{eq:apriori2} in \Cref{lemma:ExistRegSol}
 that $\phi_\e$ is uniformly bounded in $\Y$ (see \eqref{eq:Bochner}).
Observe that $\Y\subset H^1(Q)$. 
 Using the compact embedding $H^1(Q)\hookrightarrow \hookrightarrow L^2(Q)$, we conclude that there exists $\phi\in H^1(Q)$ and a subsequence $\phi_\e$ such that for $\e\to 0$,
\begin{subequations}\label{eq:convergencePhiM}
\begin{align}
& \phi_\e \rightharpoonup \phi \text{ weakly in } H^1(Q),\\
& \phi_\e \to \phi \text{ strongly in } L^2(Q).\label{eq:convergencePhiL2}
\end{align}
\end{subequations}
Setting 
\begin{align}\label{eq:MsLinfLinf}
M_s:=\Phi^{-1}(\phi)\in L^\infty(0,T;L^\infty(\Om))
\end{align}
we claim that for $\e\to 0$,
\begin{subequations}\label{eq:convergenceMse}
\begin{align}
& \p_t M_{s,\e} \rightharpoonup \p_t M_s \text{ weakly in } L^2(0,T;\Hm(\Om)),\label{eq:dtMse2dtms}\\
& M_{s,\e}\to M_s \text{ strongly in } L^2(Q).\label{eq:Mse2ms}
\end{align}
\end{subequations}
To see this, we consider a convex strictly increasing function $\eta\in C^1([0,1))$ such that
\begin{align}
\eta=\Phi \text{ in } [0,\epsilon_0),\quad \text{ and }\quad  \eta\circ\Phi^{-1}\in \mathrm{Lip}(\R^+),\label{eq:cond_eta}
\end{align}
where $\epsilon_0\in(0,1)$ was fixed in \ref{prop:D}. 
Recall that $D_0$ is strictly increasing in $[0,\epsilon_0)$, and therefore, $\Phi$ is convex in $[0,\epsilon_0)$. For $\vr>\epsilon_0$, $\Phi'(\vr)=D_0(\vr)$ is bounded away from 0, implying that $(\Phi^{-1})'(\vr)$ is bounded for $\vr>\Phi(\epsilon_0)$. Hence, it is always possible to find such a function $\eta$. 

Using \eqref{eq:MPhiLess1} and that $\eta$ is strictly increasing and convex, we obtain 
\begin{align*}
&\eta\left(\frac{1}{|Q|}\int_{Q} |M_{s,\e}-M_s|\right)
\overset{\eqref{eq:ConvexFuncIneq1}}\leq \frac{1}{|Q|}\int_{Q} \eta(|M_{s,\e}-M_s|)\overset{\eqref{eq:ConvexFuncIneq2}}
\leq \frac{1}{|Q|}\int_{Q} |\eta(M_{s,\e})-\eta(M_s)|\\
\quad \overset{\eqref{eq:cond_eta}}\lesssim & \|\Phi(M_{s,\e})-\Phi(M_s)\|_{L^1(Q)} \leq \|\Phi(M_{s,\e})-\phi_\e\|_{L^1(Q)} +\|\phi_\e-\Phi(M_s)\|_{L^1(Q)}\\
\quad \leq \;\, & \|\Phi(M_{s,\e})-\Phi_\e(M_{s,\e})\|_{L^1(Q)} +\|\phi_\e-\phi\|_{L^1(Q)}\overset{\eqref{eq:propPhieps},\eqref{eq:convergencePhiL2}}\longrightarrow 0,\quad \text{ for } \e\searrow 0.
\end{align*}
More specifically, the term $\|\Phi(M_{s,\e})-\Phi_\e(M_{s,\e})\|_{L^1(Q)}$ vanishes due to \eqref{eq:PhiepsDef} and \eqref{eq:MPhiLess1} since  for small $\e$ we observe that 
\begin{align*}
   | \Phi(M_{s,\e})-\Phi_\e(M_{s,\e})|\overset{\eqref{eq:PhiepsDef}, \eqref{eq:MPhiLess1}}= \left|\int_0^{M_{s,\e}} (D(\vr)-\max\{\e,D(\vr)\})\,\dd \vr \right|\leq C \e.
\end{align*}
Then, the continuity of $\eta$ implies that $M_{s,\e} \to M_s$ in $L^1(Q)$. Using \eqref{eq:MPhiLess1} and \eqref{eq:MsLinfLinf}, the convergence also holds in $L^2(Q)$.

Since $M_{s,\e}-h^e_0\in \X$ (see \eqref{eq:Bochner}), \eqref{eq:MPhiLess1} and $(M_{s,\e}-h^e_0)\in C([0,T];L^2(\Om)) $ from \Cref{lemma:ExistUnRegSol} also imply that $M_{s,\e}\in \W$. 
It follows from estimate \eqref{eq:apriori1} in \Cref{lemma:ExistRegSol} and \eqref{eq:MPhiLess1} that $M_{s,\e}$ is bounded in $\W$, and the bound is independent of $\e$. Hence, $M_{s,\e}$ has a weak limit in $H^1(0,T;\Hm)\cap  L^2(Q)\supset \W$. The strong convergence $M_{s,\e}\to M_s$ in $L^2(Q)$ then implies by the uniqueness of weak limits that \eqref{eq:convergenceMse} holds.

Next we show that $M_s\in \W$, i.e. it remains to show that $M_s\in C([0,T];L^2(\Om))$. From the uniform boundedness of $M_{s,\e}$ in $\W$ we conclude that $M_{s,\e}(t)\rightharpoonup M_s(t)$ weakly in $\Hm$ for almost all $t\in [0,T]$ since $(M_{s,\e}(t)-M_{s}(t),\Tilde{\f})=\int_0^t \langle \p_t(M_{s,\e}-M_{s}),\Tilde{\f}\rangle \to 0$ for all $\Tilde{\f}\in \H$, see \eqref{eq:dtMse2dtms}. Hence, for all $\f\in L^2(\Om)$ one has
\begin{align}\label{ext_lem_ex}
    |(M_{s,\e}(t)-M_s(t),\f)|&= |(M_{s,\e}(t)-M_s(t),\f-\Tilde{\f}) + (M_{s,\e}(t)-M_s(t),\Tilde{\f})|\nonumber\\
    &\leq (\|M_{s,\e}(t)\|+ \|M_s(t)\|)\|\f-\Tilde{\f}\|  + |(M_{s,\e}(t)-M_s(t),\Tilde{\f})|.
\end{align}
Let $\mu>0$ be arbitrary. Since $M_{s,\e}(t)$ and $M_s(t)$ are uniformly bounded in $L^\infty(\Om)$ by Lemma  \ref{lemma:MaxRegSol} and \eqref{eq:MsLinfLinf}, one can choose $\tilde{\f}\in \H$ such that the first term on the right hand side in \eqref{ext_lem_ex}
is less than $\mu/2$. Then, we choose $\e$ small enough such that the second term is less than $\mu/2$ implying that 
\begin{align}
    M_{s,\e}(t)\rightharpoonup M_s(t) \text{ weakly in }  L^2(\Omega) \text{ for almost all } t\in [0,T].
\end{align}
Finally to show the continuity of $M_s$ in time, we observe that for $\t>0$ one has
\begin{align*}
&\|M_s(t+\t)-M_s(t)\|^2=(M_s(t+\t)-M_{s,\e}(t+\t),M_s(t+\t)-M_s(t))\\
&\quad + (M_{s,\e}(t+\t)-M_{s,\e}(t),M_s(t+\t)-M_s(t))+ (M_{s,\e}(t)-M_{s}(t),M_s(t+\t)-M_s(t)).
\end{align*}
The first and last term on the right side are arbitrarily small for all sufficiently small $\e$ due to the weak convergence of $M_{s,\e}(t)$ in $L^2(\Omega)$. The term in the middle vanishes as $\t$ tends to zero since $M_{s,\e}\in C([0,T];L^2(\Om))$. This proves that $M_s\in  C([0,T];L^2(\Om))$.

Using \eqref{eq:convergencePhiM},\eqref{eq:convergenceMse} we can now pass to the limit in \eqref{eq:reg} and conclude that $M_s$ is a solution of \eqref{eq:unreg}. This completes the proof for the case $M_0\in H^1(\Om)$.

\textbf{(Step 2) Existence for $\mathbf{M_0\in L^1(\Om)}$ satisfying \ref{prop:IC}: } We postulate that for a given $M_0\in L^1(\Om)$ and $\mu>0$, there exists $M_{0}^\mu\in H^1(\Om)$ such that 
\begin{align}\label{eq:M0mu}
0\leq M_0^\mu \leq \bar{M}<1 \text{ a.e. in } Q, \text{ and } \|M_0^\mu-M_0\|_{L^1(\Om)}\leq \mu.
\end{align}
The existence of $\tilde{M}_0^\mu\in C^\infty_c(\R^d)$ such that  $\|\tilde{M}_0^\mu-M_0\|_{L^1(\Om)}\leq \tfrac{1}{2}\mu$ 
follows from the fact that $C^\infty_c(\R^d)$ is dense in $L^1(\R^d)$, see
Theorem 4.3 in \cite{brezis2011functional}. 
Define $M_0^\mu=\min\{\overline{M},\tilde{M}_0^\mu\}\in H^1(\Om)$, where $\overline{M}$ was defined in \ref{prop:IC}. Then $M_0^\mu$ satisfies \eqref{eq:M0mu} since
\begin{align*}
 \|M_0^\mu-M_0\|_{L^1(\Om)}&
 \leq \|M_0^\mu-\tilde{M}^\mu_0\|_{L^1(\Om)} + \|\tilde{M}_0^\mu-M_0\|_{L^1(\Om)}\\
&= \|[\tilde{M}_0^\mu-\bar{M}]_+\|_{L^1(\Om)} + \|\tilde{M}_0^\mu-M_0\|_{L^1(\Om)}\\
&\overset{\ref{prop:IC}}\leq  \|\tilde{M}_0^\mu-M_0\|_{L^1(\Om)} + \|\tilde{M}_0^\mu-M_0\|_{L^1(\Om)}\leq \mu.
\end{align*}
Now, let $M_s^\mu\in \W$ be the weak solution 
corresponding to the initial data $M^\mu_s(0)=M^\mu_0\in H^1(\Omega)$, $\mu>0$, which exists by 
Step 1. Consider a sequence $\{\mu_n\}_{n\in\N}\subset \R^+$ converging to zero. Then, the $L^1$-contraction result in \cite{otto1996l1} implies that 
there exists a constant $C>0$ independent of $\mu$ such that
\begin{align}
\|(M_{s}^{\mu_m}-M_{s}^{\mu_n})(t)\|_{L^1(\Om)}\leq C \|M_{0}^{\mu_m}-M_{0}^{\mu_n}\|_{L^1(\Om)}
\end{align}
for all  $m,n\in \N$, $t\in [0,T]$.
Note that an $L^1$-contraction result also holds for homogeneous Neumann boundary conditions, see \cite{andreianov2004uniqueness}.
Hence, $\{M_s^{\mu_n}\}_{n\in\N}$ is a Cauchy sequence in $L^1(\Om)$, and since $M_s^{\mu_n}\in L^\infty(\Om)$ is uniformly bounded with respect to $\mu_n$ (see \Cref{lemma:MaxRegSol}), it is also a Cauchy sequence in $L^2(\Om)$. Since $M_s^{\mu_n}\in C([0,T];L^2(\Om))$, we conclude that there exists $M_s\in C([0,T];L^2(\Om))$ such that  
\begin{align*}
\|M^{\mu_n}_{s}-M_s\|_{C([0,T];L^2(\Om))}\to 0\quad \text{ as } n\to \infty.
\end{align*}
The uniform boundedness of $M^{\mu_n}_{s}$ in $\W$ and $\Phi(M^{\mu_n}_{s})$ in $L^2(0,T;H^1(\Om))$ follow directly from \eqref{eq:apriori1}, \Cref{lemma:ExistRegSol}. The strong convergence of $M^{\mu_n}_{s}\to M_s$ and its uniform $L^\infty$-boundedness away from 1 (the singular point of $\Phi$), implies that $\Phi(M^{\mu_n}_{s})$ also converges strongly to $\Phi(M_s)$ in $L^2(\Om)$ for all $t\in [0,T]$.
Hence, similar to before, passing the limit $n\to \infty$ it follows that $M_s$ solves \eqref{eq:unreg}.  
\end{proof}

\subsection{A contraction argument for proving \Cref{theo:ExistUnique} }\label{sec:L1contraction}

We first show the existence of a unique weak solution 
under the additional assumptions stated in  \Cref{theo:ExistUnique} compared to \Cref{theo:Existence}.
The assumptions in \Cref{theo:ExistUnique} demand that $D_j$ depends only on $S_j$, i.e. $D_j:\R\to [D_{\min},D_{\max}]$. This is unless $\nu_j=0$ in which case we can also define $D_j$ as such. Hence, similar to \eqref{eq:defPhi}, we introduce the function
\begin{align}\label{eq:DefPhiJ}
\Phi_j(S):=\int_0^S D_j(\vr)\, \dd\vr, \qquad \text{for } j\in \{1,\dots,k\}.
\end{align}
Observe that due to \ref{prop:Df}, $\Phi_j$ is Lipschitz continuous and strictly increasing.

For a given $\vec{s}\in \Z$ let $M_s\in \W$ be the corresponding solution in \Cref{lemma:ExistUnRegSol}. Define the operator $\calA:\Z\to \Z$ such that for all $j\in \{1,\dots,k\}$, $\calA(\vec{s})_j$ satisfies $\nu_j\,\calA(\vec{s})_j\in L^2(0,T;H^1(\Om))$, $\calA(\vec{s})_j\in H^1(0,T;H^{-1}(\Om))$, and for all $\z_j\in L^2(0,T;H^1_0(\Om))$,
\begin{subequations}\label{eq:defOpA}
\begin{align}
&\int_0^T \left[\langle\z_j, \p_t \calA(\vec{s})_j\rangle_{H^1_0,H^{-1}} + \nu_j \left(\del \Phi_j(\calA(\vec{s})_j) + \bm{v}_j\,\calA(\vec{s})_j ,\del \z_j\right) \right]= \int_0^T (f_j(M_s,\vec{s}), \z_j),\label{eq:defOpAa}\\
&\text{with } \calA(\vec{s})_j(0)=S_{0,j} \text{ and } \nu_j \calA(s)_j=\nu_j h_j \text{ on $\p\Om$ in the trace sense}.
\end{align}
\end{subequations}
 To prove \Cref{theo:ExistUnique} we need the following lemma.

\begin{lemma}[$L^1$-contraction property of $\calA$]\label{lemma:L1cont}
Under the assumptions of \Cref{theo:ExistUnique}, define $\Phi_j:\R\to \R$ by \eqref{eq:DefPhiJ}. Assume that $T<T^*$ for $T^*>0$ introduced in \Cref{lemma:ExistUnRegSol}. Then the operator $\calA:\Z\to \Z$, introduced in \eqref{eq:defOpA}, is well-defined. Moreover, there exists a strictly increasing function $\Cf\in C^1(\R^+)$ with $\Cf(0)=0$ such that for all $t\in [0,T]$ and $\vec{s}_1,\, \vec{s}_2\in \Z$,
$$
\int_0^t \|\calA(\vec{s}_1)-\calA(\vec{s}_2)\|_{(L^1(\Om))^k} \leq \Cf(t) \int_0^t \|\vec{s}_1-\vec{s}_2\|_{(L^1(\Om))^k}.
$$
\end{lemma}

\begin{proof}
Since $f_j(M_s,\vec{s})\in C([0,T];L^2(\Om))$, and $D_j$ is bounded from above and below by a positive constant by \ref{prop:Df}, the existence and regularity results in \cite{alt1983quasilinear} imply that $\calA(\vec{s})_j$ is well-defined  for $\nu_j>0$ (similar to \Cref{lemma:ExistRegSol}). If $\nu_j=0$, then $\calA(\vec{s})_j$ is simply the solution of an ODE with known right hand side. From \eqref{eq:unreg}, using the $L^1$-contraction result in \cite{otto1996l1,andreianov2004uniqueness} 
and the Lipschitz continuity of $f_0$, it follows that for all $t\in [0,T]$,
\begin{align}
\|(M_{s_1}- M_{s_2})(t)\|_{L^1(\Om)}&\leq \int_0^t \|f_0(M_{s_1},\vec{s}_1)- f_0(M_{s_2},\vec{s}_2)\|_{L^1(\Om)}\nonumber\\
&\overset{\ref{prop:fg1}}\leq C_L \int_0^t \|\vec{s}_1-\vec{s}_2\|_{(L^1(\Om))^k} + C_L\int_0^t \|M_{s_1}- M_{s_2}\|_{L^1(\Om)}.
\end{align}
Applying Gronwall's Lemma \eqref{eq:Gronwall} we conclude that 
\begin{align}\label{eq:L1contractMs}
\|(M_{s_1}- M_{s_2})(t)\|_{L^1(\Om)}\leq C_L \exp(C_L\,  t)\int_0^t \|\vec{s}_1-\vec{s}_2\|_{(L^1(\Om))^k}.
\end{align}
We now apply the  $L^1$-contraction principle to \eqref{eq:defOpA} and use the Lischitz continuity of $f_j$ and the previous estimate to get 
\begin{align}\label{eq:L1contractionSeq}
\|(\calA(\vec{s}_1)_j-\calA(\vec{s}_2)_j)(t)\|_{L^1(\Om)}&\leq \int_0^t \|f_j(M_{s_1},\vec{s}_1)- f_j(M_{s_2},\vec{s}_2)\|_{L^1(\Om)}\nonumber\\
&\leq C_L \int_0^t\|\vec{s}_1-\vec{s}_2\|_{(L^1(\Om))^k} + C_L\int_0^t \|M_{s_1}- M_{s_2}\|_{L^1(\Om)}\nonumber\\
&\leq C_L(1 + C_L\,t\, \exp(C_L\,  t))\int_0^t \|\vec{s}_1-\vec{s}_2\|_{(L^1(Q))^k}.
\end{align}
Note that this estimate also holds for the case $\nu_j=0$. Hence, setting $\Cf(t) = k\,C_L\,t\,(1 + C_L\,t\, \exp(C_L\,  t))$ the result follows.
\end{proof}

\begin{proof}[\textbf{Proof of \Cref{theo:ExistUnique}}] 
Choosing $T>0$ small enough such that $\Cf(T)<1$ the existence of a unique  weak solution $(M,\vec{S})$ of \eqref{eq:main} follows from \Cref{lemma:L1cont} and
Banach's fixed point theorem. 
Since \Cref{lemma:MaxRegSol} implies that $T^*=\infty$ provided that $\G_1$ has a non-zero measure, the argument can be repeated and solutions can be patched together to cover the interval $[0,T]$ for an arbitrary $T>0$, thus concluding the proof.
\end{proof}

\subsection{A fixed point argument for proving \Cref{theo:Existence}}\label{sec:Schauder}
In this section, we use Schauder's fixed point theorem to prove the existence of solutions  for general diffusion coefficients $D_j$ satisfying \ref{prop:Df}. Since the case of ODE-PDE couplings, where $\nu_j=0$, is already covered by the previous section, here we assume that 
$\nu_j>0$. Similarly as in the previous section, we define the map $\calB: (L^2(Q))^k\to (L^2(Q))^k $ such that for all $j\in \{1,\dots,k\}$, $\calB(\vec{s})_j\in L^2(0,T;H^1(\Om))\cap H^1(0,T;H^{-1}(\Om))$, and for all $\z_j\in L^2(0,T;H^1_0(\Om))$,
\begin{subequations}\label{eq:defOpB}
\begin{align}
&\int_0^T [\langle\z_j, \p_t \calB(\vec{s})_j\rangle_{H^1_0,H^{-1}} + \nu_j (D_j(M_s,\vec{s})\del \calB(\vec{s})_j + \bm{v}_j\,\calB(\vec{s})_j ,\del \z_j) ]= \int_0^T (f_j(M_s,\calB(\vec{s})), \z_j),\label{eq:defOpBa}\\
&\text{with }\calB(\vec{s})_j(0)=S_{0,j}, \text{ and } \calB(s)_j= h_j \text{ on $\p\Om$ in the trace sense}.
\end{align}
\end{subequations}

\begin{lemma}[Schauder criteria for $\calB$]\label{lemma:Schauder}
Let $\nu_j>0$ for all $j\in \{1,\dots,k\}$. Then under the assumptions of \Cref{lemma:ExistUnRegSol}, the operator $\calB:(L^2(Q))^k\to (L^2(Q))^k$ introduced in \eqref{eq:defOpB} is well-defined, continuous, compact, and $\|\calB(\vec{s})\|_\Z$ is bounded for all $\vec{s}\in (L^2(Q))^k$.
\end{lemma}

\begin{proof} \textbf{(Step 1): Well-posedness, boundedness and compactness.} Recalling \Cref{rem:SinL2Q}, we have existence and boundedness of weak solutions $M_s\in \W$ for any $\vec{s}\in (L^2(Q))^k$.
We observe that $D_j$ satisfies the ellipticity condition by \ref{prop:Df}, $\bm{v}_j\in (L^\infty(\Om))^d$, $f_j(\cdot,s)$ is Lipschitz continuous, and for $h^e_j$ in \ref{prop:BC}, we have
$$
|f_j(M_s,h^e_j)|\overset{\ref{prop:fg1}} \leq C (1+|h^e_j|+ |M_s|)\in L^2(\Om) \text{ for a constant } C>0.
$$ 
Consequently, $\calB(\vec{s})_j \in L^2(0,T;H^1(\Om))\cap H^1(0,T;H^{-1}(\Om))$ is well-defined.
This is also evident from a Schaefer's fixed point argument \cite[Chapter 9]{evans1988partial} using the a-priori estimate obtained by inserting $\zeta_j=\calB(\vec{s})_j -h^e_j$ in \eqref{eq:defOpB}. For the first term this yields, 
\begin{align*}
&\int_0^T\langle \calB(\vec{s})_j -h^e_j, \p_t \calB(\vec{s})_j\rangle_{H^1_0,H^{-1}}=\frac{1}{2}\left[\|\calB(\vec{s}(T))_j-h^e_j\|^2-\|S_{0,j}-h^e_j\|^2\right],
\end{align*}
and for the diffusion and convection terms we obtain,  
\begin{align*}
&\nu_j\int_0^T(D_j(M_s,\vec{s})\del \calB(\vec{s})_j,\del (\calB(\vec{s})_j -h^e_j))\\
&\qquad\quad =\frac{\nu_j}{2} \int_Q D_j(M_s,\vec{s})\left[|\del \calB(\vec{s})_j|^2- |\del h^e_j|^2+ |\del (\calB(\vec{s})_j -h^e_j)|^2\right],\\
&\left|\nu_j\int_0^T(\bm{v}_j\,\calB(\vec{s})_j, \del (\calB(\vec{s})_j -h^e_j))\right|\\
&\qquad\quad \overset{\eqref{Eq:YoungsIneq}}
\leq \tfrac{\nu_j\|\bm{v}_j\|^2_{L^\infty(\Om)}}{2D_{\min}} \int_0^T \|\calB(\vec{s})_j\|^2 + \tfrac{\nu_j D_{\min}}{2}\int_0^T \|\del (\calB(\vec{s})_j -h^e_j)\|^2.
\end{align*}
Finally, we can estimate the reaction term by  
\begin{align*} 
&\int_0^T (f_j(M_s,\calB(\vec{s})),\calB(\vec{s})_j -h^e_j)= \int_0^T (f_j(M_s,\calB(\vec{s}))-f_j(M_s,\vec{h}^e)+f_j(M_s,\vec{h}^e),\calB(\vec{s})_j -h^e_j)\\
&\qquad \overset{\eqref{Eq:YoungsIneq}, \ref{prop:fg1}}\leq C\sum_{i=1}^k \int_0^T \|\calB(\vec{s})_i-h^e_i\|^2_{(L^2(\Om))^k} + \int_0^T \|f_j(M_s, \vec{h}^e)\|^2.
\end{align*}
Combining the above estimates, using Young's inequality and summing from $j=1$ to $k$, one has
\begin{align*}
&\sum_{j=1}^k\left [\|\calB(\vec{s}(T))_j-h^e_j\|^2 + \nu_j D_{\min} \int_0^T \|\del \calB(\vec{s})_j\|^2 \right ]\lesssim 1+  \int_0^T \sum_{j=1}^k\|\calB(\vec{s})_j-h^e_j\|^2.
\end{align*}
Thus, using Gronwall's Lemma \eqref{eq:Gronwall}, we conclude that $\calB(\vec{s})$ is uniformly bounded in~$L^2(0,T;(H^1(\Om))^k)$ and $L^\infty(0,T;(L^2(\Om))^k)$. Moreover, from \eqref{eq:defOpB}, we obtain 
\begin{align*}
&\|\partial_t\calB(\vec{s})_j\|_{L^2(0,T;H^{-1}(\Om))}\\
&=\sup_{\substack{\|\z_j\|_{L^2(0,T;H^1_0(\Om))}=1 }} \int_0^T \left [\nu_j (D_j(M_s,\vec{s})\del \calB(\vec{s})_j + \bm{v}_j\,\calB(\vec{s})_j ,\del \z_j)- (f_j(M_s,\calB\vec{s}), \z_j)\right ]\\
&\leq \nu_j (D_{\max} \|\del \calB(\vec{s})_j\|_{L^2(Q)} + \|\bm{v}_j\|_{L^\infty(\Om)}\|\calB(\vec{s})_j\|_{L^2(Q)}) + C_\Om\,\|f_j(M_s,\calB (\vec{s}))\|_{L^2(Q)}.
\end{align*}
Hence $\calB(\vec{s})_j-h^e_j\in L^2(0,T;H^1_0(\Om))$ and $\calB(\vec{s})_j\in H^1(0,T;H^{-1}(\Om))$ is bounded uniformly with respect to $\vec{s}\in (L^2(Q))^k$. Due to the compact embedding $L^2(0,T;H^1_0(\Om))\cap H^1(0,T;H^{-1}(\Om))\hookrightarrow \hookrightarrow L^2(Q)$ (see \cite{simon1986compact}), the mapping $\calB$ is also compact. Moreover, due to the continuous embedding  $L^2(0,T;H^1_0(\Om))\cap H^1(0,T;H^{-1}(\Om))\subset C([0,T];L^2(\Om))$, $\|\calB(\vec{s})\|_\Z$ is bounded uniformly.

\textbf{(Step 2): Continuity.} Let the sequence $\{\vec{s}^i\}_{i\in \N}$ converge to $ \vec{s}^{\,\star}$ in $(L^2(Q))^k$. Let $M_{s}^i$ and $M^\star_s$ denote the solutions of \eqref{eq:unreg} for $\vec{s}=\vec{s}^i$ and $\vec{s}=\vec{s}^{\,\star}$ respectively. Then, by using the $L^1$-contraction result \eqref{eq:L1contractMs} we have that 
$\|(M_{s}^i-M_s^\star)(t)\|_{L^1(\Om)}\to 0$ for all $t\in (0,T]$. Since, $M_{s}^i$ are bounded in $L^\infty(\Om)$ (\Cref{lemma:ExistUnRegSol}), one further has that 
\begin{align}\label{eq:MSto0inL2}
\|\vec{s}^i-\vec{s}^{\,\star}\|_{(L^2(Q))^k}+ \|M_{s}^i-M_s^\star\|_{C([0,T];L^2(\Om))}\rightarrow 0\ \text{ as } i\rightarrow \infty.
\end{align}
Observe that, since $\calB(\vec{s}^{\,\star})_j \in L^2(0,T;H^1(\Om))$, for any given $\e>0$, there exists  $s_j^{\e,\star}\in C^\infty(\R^{d})$ such that 
\begin{align}\label{eq:Sstar_eps_def}
    \|\calB(\vec{s}^{\,\star})_j - s_j^{\e,\star}\|_{L^2(0,T;H^1(\Om))} \leq \e/4D_{\max}.
\end{align}

We consider the difference of   $\calB(\vec{s}^i)_j$ and $\calB(\vec{s}^{\,\star})_j$ by subtracting two versions of \eqref{eq:defOpB}.
First, we split up the diffusion term, 
\begin{subequations}\label{eq:DTermsplit}
\begin{align}
&\int_0^T (D_j(M_s^i,\vec{s}^i)\del \calB(\vec{s}^i)_j- D_j(M_s^{\star},\vec{s}^{\,\star})\del \calB(\vec{s}^{\,\star})_j,\del \zeta_j)\nonumber\\
&=\int_0^T (D_j(M_s^i,\vec{s}^i)\del (\calB(\vec{s}^i)_j-\calB(\vec{s}^{\,\star})_j)+ (D_j(M_s^i,\vec{s}^i)- D_j(M_s^{\star},\vec{s}^{\,\star}))\del \calB(\vec{s}^{\,\star})_j,\del \zeta_j),
\end{align}
and use H\"older's inequality to estimate the second term as follows 
\begin{align}
&\|(D_j(M_s^i,\vec{s}^i)- D_j(M_s^{\star},\vec{s}^{\,\star}))\del \calB(\vec{s}^{\,\star})_j\|_{L^2(Q)}\nonumber \\
&\leq  \|(D_j(M_s^i,\vec{s}^i)- D_j(M_s^{\star},\vec{s}^{\,\star}))\del s^{\e,\star}_j\|_{L^2(Q)} + \|(D_j(M_s^i,\vec{s}^i)- D_j(M_s^{\star},\vec{s}^{\,\star}))\del (\calB(\vec{s}^{\,\star})_j - s^{\e,\star}_j)\|_{L^2(Q)}\nonumber\\
& \leq  \|s^{\e,\star}_j\|_{C^1(Q)} \|D_j(M_s^i,\vec{s}^i)- D_j(M_s^{\star},\vec{s}^{\,\star})\|_{L^2(Q)} + 2 D_{\max} \|\del (\calB(\vec{s}^{\,\star})_j - s^{\e,\star}_j))\|_{L^2(Q)}\overset{\eqref{eq:MSto0inL2}, \eqref{eq:Sstar_eps_def}} \leq \e,
\end{align}
for all $i\geq i_{\e,1}$, where $i_{\e,1}\in \N$ is large enough. Here we used the Lipschitz continuity of $D_j$, see \ref{prop:Df}.
Furthermore \eqref{eq:MSto0inL2} implies for $i\geq i_{\e,2}$, where $i_{\e,2}\in \N$ is large enough, that
\begin{align}
\|f_j(M_s^i,\calB(\vec{s}^i))- f_j(M_s^{\star},\calB(\vec{s}^{\,\star}))\|_{L^2(\Om)}\overset{\ref{prop:fg1}}\lesssim \e+\|\calB(\vec{s}^i)-\calB(\vec{s}^{\,\star})\|_{(L^2(\Om))^k}.
\end{align}
\end{subequations}
 Hence, defining $\eth s_j^i:=\calB(\vec{s}^i)_j-\calB(\vec{s}^{\,\star})_{j}$, it follows from \eqref{eq:defOpB} that for $i\geq \max\{i_{\e,1},i_{\e,2}\}$,
\begin{align*}
&\left |\int_0^T [\langle\z_j, \p_t \,\eth s^i_j\rangle_{H^1_0,H^{-1}} + \nu_j (D_j(M_s^i,\vec{s}^i)\del \eth s^i_j + \bm{v}_j\,\eth s^i_j ,\del \z_j) ]\right |\\
&\overset{\eqref{eq:DTermsplit}}
\lesssim  \int_0^T \left [\sum_{j=1}^k\|\eth s^i_j\|\|\z_j\|+\e\|\z_j\| + \e \|\del \z_j\|\right ]\\
&\ \, \lesssim \ \int_0^T \left [\sum_{j=1}^k\|\eth s^i_j\|\|\z_j\| +\varepsilon+\varepsilon\left(\|\z_j\|^2 + \|\del \z_j\|^2\right)\right ].
\end{align*}
Finally, we insert the test function $\z_j=\eth s^i_j\in L^2(0,T;H^1_0(\Om))$ and sum up the resulting estimates from $j=1$ to $k$. Then for $\e>0$ small enough, one obtains using 
Gronwall's lemma \eqref{eq:Gronwall}
\begin{align*}
\sum_{j=1}^k\left [\|\eth s^i_j\|^2 + \int_0^T \|\del \eth s^i_j\|^2 \right ]\lesssim \e.
\end{align*}
Hence, the right hand side can be made arbitrary small by choosing $\e>0$ small enough which simply requires $i\geq \max\{i_{\e,1},i_{\e,2}\}$. Passing to the limit $\e\to 0$ we conclude that $\|\eth s^i_j\|_{L^2(0,T;H^1(\Om))}+ \|\eth s^i_j\|_\Z\to 0$ for all $j=1,\dots,k$.
 This shows that the operator $\calB$ is  continuous, thus, concluding the proof.
\end{proof}

\begin{proof}[\textbf{Proof of \Cref{theo:Existence}}] If $\nu_j>0$ for all $j\in \{1,\dots,k\}$, then using Lemma \ref{lemma:Schauder} and
Schauder's fixed point theorem, see \cite[Chapter 9]{evans1988partial},  we conclude that a fixed point $\vec{s}=\vec{S}\in \Z\subset (L^2(Q))^k$ exists of the mapping $\calB$, i.e. $\calB(\vec{S})=\vec{S}$. It is easy to verify that this fixed point $(M_S,\vec{S})$ is a weak solution of \eqref{eq:main}. 

If $\nu_{\bar{j}}=0$ for all $\bar{j}\in\mathcal{I}\subset\{1,\dots,k\}$, then the theorem is proved by first applying the contraction mapping in \Cref{lemma:L1cont} for \eqref{eq:M} and \eqref{eq:Sj} with $\bar{j}\in\mathcal{I}$, followed by applying the fixed point argument for $j\in\{1,\dots,k\}\setminus\mathcal{I}$. The details are left for the avid reader.
\end{proof}

The approach developed in this section can be extended to systems with degenerate diffusion coefficients $D_j$ under some additional assumptions. Below we discuss an example of such a case.

\begin{corollary}[Existence of weak solutions for degenerate $D_j$]\label{cor:degenerateDj} Assume that \ref{prop:D}, \ref{prop:fl}--\ref{prop:BC} and \ref{prop:fg2} hold.  
For some $\j\in \{1,\dots,k\}$, instead of \ref{prop:Df}, assume that the diffusion coefficient $D_\j$  satisfies $D_\j(m,\vec{s})=D_\j(s_\j)$ where
\begin{align*}
&D_\j:[0,\infty)\to\R \text{ is continuous and strictly increasing, with } D_\j(0)=0.
\end{align*}
Moreover, let $\mathrm{ess} \inf \{S_{0,\j}\}\geq 0, \mathrm{ess} \inf \{h_\j\}\geq 0$, $\bm{v}_\j= \bm{0}$ in $Q$, $f_\j(\cdot,\vec{s})\leq f^\j_{\max}(s_j)$ for some function $f^\j_{\max}\in \mathrm{Lip}(\R^+)$ and $f_\j(m,\vec{s})\geq 0$ if $s_\j=0$.

Then a weak solution $(M,\vec{S})$ of \eqref{eq:main} exists  in the sense of \Cref{def:WeakSol}, but with $\nu_\j\,S_\j\in L^2(0,T;H^1(\Om))$ replaced by
$\nu_\j\int_0^{S_\j} D_\j\in L^2(0,T;H^1(\Om))$.
The solution is unique if for all $j\in \{1,\dots,k\}$ either $\nu_j=0$ or $D_j$ depends only on $s_j$. 
Furthermore, $S_\j$ is non-negative and bounded almost everywhere in $Q$.
\end{corollary}

If $D_\j=D_\j(s_\j)$ is degenerate for some $\j\in \{1,\dots,k\}$, without loss of generality we assume $\nu_\j>0$. We define $\Phi_\j$ as in \eqref{eq:DefPhiJ} and fix all $s_j\in C([0,T];L^2(\Om))$, $j\not=\j$, for $s_\j\in C([0,T];L^2(\Om))$. Let $M_s\in \W$ be the solution of \eqref{eq:unreg} from \Cref{lemma:ExistUnRegSol}.
Then, $\tilde{s}_\j\in H^1(0,T;H^{-1}(\Om))\cap C([0,T];L^2(\Om))$ with $\Phi_\j(\tilde{s}_\j)\in L^2(0,T;H^1(\Om))$ is defined as the solution of the following problem, for all $\z_\j\in L^2(0,T;H^1_0(\Om))$,
\begin{align*}
&\int_0^T [\langle\z_\j, \p_t \tilde{s}_\j\rangle_{H^1_0,H^{-1}} + \nu_\j (\del \Phi_\j(\tilde{s}_\j),\del \z_\j) ]= \int_0^T (f_\j(M_s,(s_1,\dots,\tilde{s}_\j,\dots,s_k)), \z_\j),\\
&\text{with } \tilde{s}_\j(0)=S_{0,\j} \  \text{ and } \  \Phi_\j (\tilde{s}_\j)=\Phi_\j( h_\j )\text{ on $\p\Om$ in the trace sense}.
\end{align*}
The existence and uniqueness of $\tilde{s}_\j$ then follows from \Cref{lemma:ExistRegSol,lemma:ExistUnRegSol}. Defining $\hat{S}_\j(t):=\bar{S}+  \int_0^t f^\j_{\max}$, we have, similar to \Cref{lemma:MaxRegSol} that $0\leq \tilde{s}_\j(t)\leq \hat{S}_\j(t)$ a.e. in $\Om$ for all $t>0$.    Following \Cref{lemma:L1cont}, we further conclude that $\tilde{s}_\j$ satisfies an $L^1$-contraction result with respect to $s_\j$ since all other $s_j$, $j\not=\j$, are fixed. Finally,  the arguments in the proof of \Cref{theo:ExistUnique} concludes the proof.

\section{Homogeneous Neumann boundary conditions}\label{sec:Neumann}

In this section, we show the existence of solutions for homogeneous Neumann  boundary conditions and present the proof of \Cref{lemma:ExistRegSol}. The global existence of solutions cannot be guaranteed for homogeneous Neumann boundary conditions, since the density $M$ might reach 1 in finite time. The local existence and uniqueness of solutions is analyzed in \Cref{sec:ExistNeumann} and the
finite time blow-up in \Cref{sec:blowup}. 

\subsection{Existence of weak solutions}\label{sec:ExistNeumann}

\begin{theorem}[Local well-posedness for homogeneous Neumann conditions]\label{theo:Neumann}
 Let $\G_1=\emptyset$. We assume that \ref{prop:D}--\ref{prop:BC} and \ref{prop:fg2} hold.
 Then, there exists a positive time $T^*\geq \sup\{t: \hat{M}(t)<1\}>0$, where $\hat{M}\in C^1(\R^+)$ is defined in \Cref{lemma:MaxRegSol}, such that for $T\in(0,T^*)$, a weak solution $(M,\vec{S})$ of \eqref{eq:main} exists in the sense of \Cref{def:WeakSol}. Moreover, the solution is unique if either $\nu_j=0$, or $D_j$ depends only on $S_j$ for all $j\in \{1,\dots,k\}$. 
\end{theorem}

This result essentially follows from the proof of \Cref{theo:ExistUnique,theo:Existence}. Indeed, note that  \Cref{lemma:MaxRegSol,lemma:ExistUnRegSol,lemma:L1cont,lemma:Schauder} were proven for the general case, i.e., they also hold for homogeneous Neumann boundary conditions. Hence, it remains to show \Cref{lemma:ExistRegSol} for the case that $\Gamma_1=\emptyset$. In this subsection, we present the proof  under more general assumptions that cover mixed as well as homogeneous Neumann boundary conditions. The proof follows the Rothe method \cite{kacur1986method} that is based on time-discrete approximations of the solutions. To simplify the notation for different boundary conditions (see \eqref{eq:HisH1})
\[\text{without loss of generality we \textbf{assume that} } h_0\equiv 0, \text{ implying } h_0^e\equiv 0.\]

\subsubsection{Well-posedness of backward Euler time--discretizations}
 
 We consider an equivalent formulation of \eqref{eq:reg} and discrtize it using the backward Euler scheme. Following \eqref{eq:propPhieps}, we introduce
\begin{align}
\b_\e:={\Phi_\e}^{-1}\quad \text{ such that }\quad \e\leq {\b_\e}'\leq \e^{-1}. \label{def:BetaEps}
\end{align}
Then replacing $\Phi_\e(M_{s,\e})$ by $u$ and $M_{s,\e}$ by $\b_\e(u)$, we demand that  $u\in L^2(0,T;\H)$ with $\b_\e(u)\in H^1(0,T;\Hm)$ and $\b_\e(u(0))=M_0$ satisfies
\begin{align}
\int_0^T \langle \f,\p_t \b_\e(u) \rangle + \int_0^T (\del u,\del \f)= \int_0^T (f_0(\b_\e(u),\vec{s}),\f) \quad \text{forall}\ \f\in L^2(0,T;\H)
\end{align}
and a given $\vec{s}\in \Z$. For $N\in \N$, we denote by  $\t:=T\slash N$ the time-step size and set $t_n:=n \t$ for $n\in \{0,1,\dots,N\}$. 
Then we define the time--discrete sequence $\{u_n\}_{n=1}^{N}\subset \H$ recursively as follows: 
setting $u_0:=\Phi_\e(M_0)$ (i.e., $\b_\e(u_0)=M_0$), let $u_n\in \H$ be the solution of 
\begin{align}\label{eq:DefTimeDiscrete}
\tfrac{1}{\t}(\b_\e(u_n)-\b_{\e}(u_{n-1}),\z) + (\del u_n,\del \z)= (f_0(\b_\e(u_n),\vec{s}(t_n)),\z)\qquad \text{for all } \z\in \H.
\end{align}
The following lemma implies the well-posedness of the time--discrete formulation.

\begin{lemma}[Well-posedness a semilinear elliptic problem] 
For a given $F\in L^2(\Om)$, there exists a unique solution $w\in \H$ of the elliptic problem 
\begin{align}
(\b_\e(w),\z) + (\del w,\del \z)= (F,\z)\qquad \text{for all } \z\in \H.
\end{align}\label{lemma:SemLinU}
\end{lemma}

\begin{proof}
The proof is based on monotonicity arguments. Let the operator $\calF: \H\to \Hm$ be defined by the inner product
\begin{align}
 \langle \z,\calF(w)\rangle:= (\b_\e(w),\z) + (\del w,\del \z).
 \end{align}
 Then $\calF$ is strongly monotone since
$$
  \langle w-v,\calF(w)-\calF(v)\rangle \geq  \e\|w-v\|^2+ \|\del (w-v)\|^2\geq \e \|w-v\|_{\H}^2.
$$
 Furthermore,  $\calF$ is Lipschitz continuous since, using the Cauchy-Schwarz inequality, 
 we obtain
  $$\|\calF(w)-\calF(v)\|_{\Hm}\leq \sup_{\z\in \H}\left (\frac{\e^{-1}\|w-v\|\|\z\|+ \|\del (w-v)\|\|\del \z\|}{\|\z\|_{\H}}\right )\leq \e^{-1} \|w-v\|_{\H}.$$
  Hence, invoking the nonlinear Lax-Milgram Lemma \cite[Theorem 2.G]{zeidler1995applied} completes the proof.
\end{proof}

Observe that the operator $\widetilde{\calF}: \H\to \Hm$ defined by the inner product
\begin{align}
 \langle \z,\widetilde{\calF}(w)\rangle:= (\b_\e(w)-\t\,f_0(\b_\e(w),\vec{s}(t_n)),\z) + (\del w,\del \z).
 \end{align}
is strictly monotone with respect to $w$ if $\t<C_L^{-1}$ by  \ref{prop:fg1} and Lipschitz continuous. Hence, adjusting the arguments in the proof of \Cref{lemma:SemLinU} to \eqref{eq:DefTimeDiscrete} we obtain the existence and uniqueness of the time--discrete solutions. 

\begin{lemma}[Well-posedness of the time--discrete solutions] 
Let \ref{prop:D}--\ref{prop:BC} hold.
Then the sequence $\{u_n\}_{n=1}^{N}\subset \H$ introduced in   \eqref{eq:DefTimeDiscrete} is well-defined for $\t<C_L^{-1}$. 
\end{lemma}
 
 \subsubsection{Interpolations in time}
 
For a fixed $N\in \N$ with $\t=T\slash N$, we define the time interpolates $\hat{u}_\t\in L^\infty(0,T;\H)$ and  $\bar{u}_\t\in C([0,T];\H)$ from the time--discrete solutions $\{u_n\}_{n=1}^{N}\subset \H$ such that for $t\in (t_{n-1},t_n]$, $n\in \{1,\dots,N\}$,
\begin{align}\label{eq:TimeInterpol}
\hat{u}_\t:=u_n, \quad \text{ and } \quad \bar{u}_\t:=\b_\e^{-1}\left(\b_\e(u_{n-1}) + \tfrac{t-t_{n-1}}{\t} (\b_\e(u_{n})-\b_\e(u_{n-1}))\right).
\end{align}
 Observe that $\bar{u}_\t$ satisfies for all $n\in \{1,\dots,N\}$,
 \begin{align}\label{eq:prop_ubar}
     \bar{u}_\t(t_n)=u_n,\;\; \text{ and } \;\;\p_t \b(\bar{u}_\t)=\frac{\b_\e(u_n)-\b_\e(u_{n-1})}{\t}\quad \text{ for } t\in (t_{n-1},t_n].
 \end{align}

\begin{lemma}[Uniform boundedness of the time interpolates with respect to $\t$]\label{lemma:UnifromBoundUn}
Let \ref{prop:D}--\ref{prop:BC} hold. 
Then there exist constants $\t^*, C>0$, independent of $\t$, such that for $\t<\t^*$,
\begin{subequations}\label{eq:UnifromBoundUn}
\begin{align}\label{eq:UnifromBoundUhat}
\|\b_\e(\hat{u}_\t)\|^2 + \int_0^T \|\del \hat{u}_\t\|^2 &\leq C+ C\int_0^T \left (\|\vec{s}\|^2+ \| \hat{u}_\t\|^2\right ),\\
\|\b_\e(\bar{u}_\t)\|^2 + \int_0^T[\|\del \bar{u}_\t\|^2 + \|\p_t \b_\e(\bar{u}_\t)\|_{\Hm}^2 ] &\leq  C+ C\int_0^T \left (\|\vec{s}\|^2+ \| \hat{u}_\t\|^2 \right ).\label{eq:UnifromBoundUbar}
\end{align}
\end{subequations}
The above inequalities imply the uniform boundedness of $\b_\e(\hat{u}_\t),\, \b_\e(\bar{u}_\t)\in L^\infty(0,T;L^2(\Om))$, $\hat{u}_\t,\, \bar{u}_\t\in L^2(0,T;\H)$ and  $\b_\e(\bar{u}_\t)\in H^1(0,T;\Hm)$ with respect to $\t<\t^*$.
\end{lemma}


\begin{proof}
\textbf{(Step 1) Uniform boundedness of $\|\b_\e(u_\t)\|$:} We choose the test function $\z=\b_\e(u_n)\in \H$ in \eqref{eq:DefTimeDiscrete}, yielding 
\begin{align}
(\b_\e(u_n)-\b_{\e}(u_{n-1}),\b_\e(u_n)) + \t (\del u_n,\del \b_\e(u_n))= \t (f_0(\b_\e(u_n),\vec{s}(t_n)),\b_\e(u_n)).
\end{align}
Observe from the identity $2a(a-b)=a^2-b^2+(a-b)^2$ that 
$$
(\b_\e(u_n)-\b_{\e}(u_{n-1}),\b_\e(u_n)) =\tfrac{1}{2}[ \|\b_\e(u_n)\|^2-\|\b_\e(u_{n-1})\|^2+ \|\b_\e(u_n)-\b_\e(u_{n-1})\|^2].
$$
Moreover, one has for some constant $C'>0$ independent of $\e$ and $\t$ that
\begin{align*}
  (f_0(\b_\e(u_n),\vec{s}(t_n)),\b_\e(u_n))&\overset{\ref{prop:fg1}}\leq C'(1 + \|\vec{s}(t_n)\|^2+\|\b_\e(u_n)\|^2),\\
  (\del u_n,\del \b_\e(u_n))&\geq \e \|\del u_n\|^2.
\end{align*}
Then, combining these inequalities we obtain
$$
\|\b_\e(u_n)\|^2 +\|\b_\e(u_n)-\b_\e(u_{n-1})\|^2+  \e \t \|\del u_n\|^2\leq \|\b_\e(u_{n-1})\|^2 + \t C'\left (1 +\|s_n\|^2+ \|\b_\e(u_n)\|^2\right ).
$$
Applying the discrete Gronwall Lemma \eqref{eq:discGronwall} for small enough $\t>0$, we have for a constant $C>0$ independent of $N$ or $\e$ that 
\begin{align}\label{eq:ImpBeUn}
\|\b_\e(u_N)\|^2 + \sum_{n=0}^N [\e \|\del u_n\|^2 \t + \|\b_\e(u_n)-\b_\e(u_{n-1})\|^2]  \leq \|\b_\e(u_0)\|^2+ C+C \sum_{n=0}^N  \|\vec{s}(t_n)\|^2\t.
\end{align}
For $\t>0$ small enough, one can estimate 
\begin{align}\label{eq:intStn_bounded_intS}
\sum_{n=0}^N  \|\vec{s}(t_n)\|^2\t\leq \left( 1+ \int_0^T \|\vec{s}\|^2\right).
\end{align}
Combining \eqref{eq:ImpBeUn}-\eqref{eq:intStn_bounded_intS} we conclude that $\b_\e(u_N)$ and, in extension of the method,  
all $\b_\e(u_n)$ are uniformly bounded in $L^2(\Om)$  with respect to $N$ and $\e$. Then, the definition \eqref{eq:TimeInterpol} implies that $\b_\e(\hat{u}_\t)$ ($=\b_\e(u_n)$ for $t\in (t_{n-1},t_n]$) and $\b_\e(\bar{u}_\t)$ ($\leq \max\{\b_\e(u_n),\b_\e(u_{n-1})\}$ for $t\in (t_{n-1},t_n]$) are uniformly bounded.


\textbf{(Step 2) Uniform boundedness of $\|\del u_\t\|_{L^2(0,T;\H)}$:} Let us now test \eqref{eq:DefTimeDiscrete} with $\z=u_n \in \H$. This yields
\begin{align}
(\b_\e(u_n)-\b_{\e}(u_{n-1}),u_n) + \t \|\del u_n\|^2= \t (f_0(\b_\e(u_n),\vec{s}(t_n)),u_n).
\end{align}
Now, from the convexity of the function $\int_0^m \Phi_\e$ (see \eqref{eq:propPhieps}), one has 
\begin{align*}
\int_{\b_\e(u_{n-1})}^{\b_\e(u_{n})} \Phi_\e \leq \Phi_\e(\b_\e(u_n))(\b_\e(u_n)-\b_{\e}(u_{n-1})) = u_n (\b_\e(u_n)-\b_{\e}(u_{n-1})).
\end{align*}
For the last term, we observe that 
\begin{align}\label{eq:f0un_estimate}
&(f_0,u_n)\overset{\eqref{Eq:YoungsIneq}}
\leq \frac{1}{2}[\|f_0(\b_\e(u_n),\vec{s}(t_n))\|^2 + \|u_n\|^2]\overset{\ref{prop:fg1}}
\leq C[1+ \|\b_\e(u_n)\|^2 + \|\vec{s}(t_n)\|^2+ \|u_n\|^2].
\end{align}
Hence, summing the inequalities from $n=0$ to $n=N$, using the uniform boundedness of $\|\b_\e(u_n)\|$ from \eqref{eq:ImpBeUn},  and 
\[
\sum_{n=1}^N \int_{\b_\e(u_{n-1})}^{\b_\e(u_{n})} \Phi_\e = \int^{\b_\e(u_{N})}_{0} \Phi_\e- \int^{\b_\e(u_0)}_{0} \Phi_\e,
\]
the estimate becomes
\begin{align}
\int_{\Om}\int^{\b_\e(u_{N})}_{0} \Phi_\e + \sum_{n=0}^N \|\del u_n\|^2\t\leq  \int_{\Om}\int^{M_0}_{0} \Phi_\e + C\,T + C\sum_{n=0}^N  \left(\|\vec{s}(t_n)\|^2 + \|u_n\|^2\right )\,\t.
\end{align}
Hence, noting that $\int_{0}^{T} \|\hat{u}_{\t}\|^2=\sum_{n=0}^N \tau\|u_n\|^2$ we have that  $\int_{0}^{T} \|\del \hat{u}_{\t}\|^2=\sum_{n=1}^{N} \|\del u_n\|^2\, \t$ is bounded as stated in \eqref{eq:UnifromBoundUhat}, and correspondingly, following its definition, the other interpolate $\bar{u}_{\t}$ is also bounded in $L^2(0,T;\H)$ as in \eqref{eq:UnifromBoundUbar}.

\textbf{(Step 3) Uniform boundedness of $\|\p_t \b_\e(\bar{u}_\t)\|_{L^2(0,T;\Hm)}$:} We have from \eqref{eq:DefTimeDiscrete} that 
\begin{align*}
&\|\p_t \b_\e(\bar{u}_{\t})\|_{\Hm}\overset{\eqref{eq:prop_ubar}}=\|\tfrac{1}{\t}(\b_\e(u_n)-\b_\e(u_{n-1}))\|_{\Hm}= \sup_{\z\in \H} \frac{\frac{1}{\t} \langle \b_\e(u_n)-\b_{\e}(u_{n-1}),\z\rangle}{\|\z\|_{\H}}\\
&= \sup_{\z\in \H} \frac{-(\del u_n,\del \z)+ (f_0(\b_\e(u_n),\vec{s}(t_n)),\z)}{\|\z\|_{\H}}\leq \|\del u_n\| + C(1+\|\vec{s}(t_n)\|+\|\b_\e(u_n)\|).
\end{align*}
The bound in \eqref{eq:UnifromBoundUbar} for $\p_t \b_\e(\bar{u}_{\t})$ now follows from Steps 1 and 2. 

Observe that, since $\b_\e$ satisfies  \eqref{def:BetaEps}, the $\int_{0}^{T} \|\hat{u}_{\t}\|^2$ terms on the right hand side of \eqref{eq:UnifromBoundUhat}--\eqref{eq:UnifromBoundUbar} can be bounded above using the Gronwall Lemma in \eqref{eq:UnifromBoundUhat}, which yields the uniform boundedness of the quantities stated in \eqref{lemma:UnifromBoundUn} in their respective spaces with respect to $\t$. However, note that the bounds may still depend on $\e>0$.
\end{proof}

\begin{lemma}[Higher regularity of the time interpolates for $M_0\in \H$] 
Let the assumptions of Lemma \ref{lemma:UnifromBoundUn} hold.
If, in addition $M_0\in \H$, then for a constant $C>0$ independent of $\t$, one has 
\begin{align}
\|\del \hat{u}_\t(t)\|^2+  \tfrac{\e}{2} \int_0^T \|\p_t \b_\e(\bar{u}_\t)\|^2\leq \|\del \Phi_\e(M_0)\|^2 +C+ C\int_0^T \|\vec{s}\|^2.
\end{align}\label{lemma:HigherReg}
\end{lemma}
\begin{proof} We insert the test function $\z=u_n-u_{n-1}$ in \eqref{eq:DefTimeDiscrete}. This gives term-wise 
\begin{align}
( \b_\e(u_n)-\b_\e(u_{n-1}),u_n-u_{n-1}) &\overset{\eqref{def:BetaEps}}\geq \e\t^2 \int_{\Om} \left|\frac{\b_\e(u_n)-\b_\e(u_{n-1})}{\t}\right|^2\overset{ \eqref{eq:prop_ubar}}= \t^2\e \|\p_t \b_\e(\bar{u}_\t)\|^2,\nonumber\\
\t(\del u_n, \del (u_n-u_{n-1}))&=\tfrac{\t}{2}[\|\del u_n\|^2 -\|\del u_{n-1}\|^2 + \|\del (u_n-u_{n-1})\|^2],\nonumber\\
 \t(f_0,u_n-u_{n-1})&\leq \tfrac{\t^2}{2\e^3} \|f_0\|^2 + \tfrac{\e^3}{2} \|u_n-u_{n-1}\|^2
 \overset{\eqref{def:BetaEps}, \eqref{eq:prop_ubar}}\leq \tfrac{\t^2}{2\e^3} \|f_0\|^2 + \tfrac{\e\t^2}{2} \|\p_t \b_\e(\bar{u}_\t)\|^2.\nonumber
\end{align}
Similarly to \eqref{eq:f0un_estimate}, we obtain
\begin{align*}
\|f_0\|^2\leq C(1+\|\vec{s}(t_n)\|^2+ \|\b_\e(u_n)^2\|)\leq C\left [1+\|\vec{s}(t_n)\|^2+ \sum_{n=1}^N\, \|\vec{s}(t_n)\|^2\t\right ],
\end{align*}
where we used Lemma \ref{lemma:UnifromBoundUn}.
Finally, summing the resulting inequalities from $n=1$ to $n=N$ and cancelling out $\t$ one has
\begin{align}
\|\del u_N\|^2 + \tfrac{\e}{2} \sum_{n=1}^N \|\p_t \b_\e(\bar{u}_\t)\|^2 \t \leq \|\del \Phi_\e(M_0)\|^2 + C\sum_{n=1}^N  (1+ \|\vec{s}(t_n)\|^2) \t,
\end{align}
 which proves the lemma.
\end{proof}

\begin{remark}[Covering homogeneous Neumann condition]
The above lemmas cover both homogeneous mixed boundary conditions and homogeneous Neumann conditions. In the latter case, $\H=H^1(\Om)$. To cover the case of inhomogeneous mixed boundary conditions, we have to test with $\z=\b_\e(u_n)-h^e_0$ in Step 1 of \Cref{lemma:UnifromBoundUn} and with $\z=u_n-\Phi_\e(h^e_0)$ in Step 2. The details are straightforward, and hence, omitted.
\end{remark}

\subsubsection{Proof of \Cref{lemma:ExistRegSol}}
\textbf{(Step 1) Existence:}
Note that $\b_\e$ is Lipschitz  and strictly increasing by \eqref{def:BetaEps}. Using this fact and applying Gronwall's Lemma to \eqref{eq:UnifromBoundUhat} implies that $\hat{u}_\t,\, \b_\e(\hat{u}_\t)$ are uniformly bounded in $ L^\infty(0,T;L^2(\Om))$ with respect to $\t$. Consequently $\bar{u}_\t,\, \b_\e(\bar{u}_\t)$ are uniformly bounded as well by \eqref{eq:UnifromBoundUbar}. Thus, using \eqref{eq:UnifromBoundUbar} we obtain the uniform boundedness of $\b_\e(\bar{u}_\t)\in \X$.  Due to the compact embedding of $L^2(Q)$ in $\X$ \cite{simon1986compact}, there exists $u\in \X$ such that along a subsequence $\t\to 0$, 
\begin{subequations}\label{eq:ConvergencesUbarUhat}
\begin{align}
&\b_\e(\bar{u}_{\t})\rightharpoonup \b_\e(u) \text{ weakly in } \X=L^2(0,T;\H)\cap H^1(0,T;\Hm),\\
&\b_\e(\bar{u}_{\t})\to \b_\e(u) \text{ strongly in } L^2(Q).
\end{align}
Using \eqref{eq:ImpBeUn}, one has
\begin{align*}
& \int_0^T \|\b_\e(\hat{u}_\t)-\b_\e(\bar{u}_\t)\|^2=\sum_{n=1}^{N}\int_{t_{n-1}}^{t_n}\left\|\tfrac{t_n-t}{{\t}}(\b_\e(u_n)-\b_\e(u_{n-1}))\right\|^2 \dd t\\
&=\sum_{n=1}^{N} \left\|(\b_\e(u_n)-\b_\e(u_{n-1}))\right\|^2 \int_{t_{n-1}}^{t_n} \left (\tfrac{t_n-t}{{\t}}\right )^2 \dd t=\tfrac{{\t}}{3}  \sum_{n=1}^{N}\|\b_\e(u_n)-\b_\e(u_{n-1})\|^2  \to 0.
\end{align*} 

This, along with the uniform bound with respect to $\tau$ of $\hat{u}_\t,\, \bar{u}_\t$ in $L^2(0,T;\H)$ in \eqref{eq:UnifromBoundUn} and the strict monotonicity of $\b_\e$ in \eqref{def:BetaEps} implies that
\begin{align}
&\hat{u}_{\t},\,\bar{u}_\t \rightharpoonup u \text{ weakly in } L^2(0,T;\H),\\
&\hat{u}_{\t},\,\bar{u}_\t\to u \text{ strongly in } L^2(Q).
\end{align}
Observe from \eqref{eq:DefTimeDiscrete} and \eqref{eq:TimeInterpol} that $\hat{u}_\t$ and $\bar{u}_\t$ satisfy 
\begin{align}
\int_0^T [\langle\z, \p_t \b_\e(\bar{u}_\t) \rangle + (\del \hat{u}_\t,\del \z)] =\int_0^T (f_0(\b_\e(\hat{u}_\t),\vec{s}),\z),
\end{align}
for all $\z\in C([0,T];\H)$ which is dense in $L^2(0,T;\H)$. Passing to the limit $\t\to 0$ we conclude from \eqref{eq:ConvergencesUbarUhat} that $u$ solves the system 
$$
\int_0^T [\langle\z, \p_t \b_\e(u) \rangle + (\del u,\del \z) ]=\int_0^T (f_0(\b_\e(u),\vec{s}),\z).
$$
Defining $M_{s,\e}=\b_\e(u)$ we obtain the desired solution. We conclude that
$M_{s,\e}\in \X\hookrightarrow C([0,T];L^2(\Omega))$, see \cite[Section 5.9]{evans1988partial} for the continuous embedding result.

\textbf{(Step 2) A-priori bounds:} The a-priori estimate \eqref{eq:apriori1} follows by inserting $\f=M_{s,\e}$ and $\f=\Phi_\e(M_{s,\e})$ in \eqref{eq:reg} and proceeding similar to the steps of the time-discrete case in \Cref{lemma:UnifromBoundUn}. \Cref{lemma:HigherReg} shows that if in addition $M_0\in H^1(\Om)$ then $\p_t M_{s,\e}\in L^2(Q)$ which implies by the definition of weak derivatives that 
$$
\D \Phi_\e(M_{s,\e})=(\p_t M_{s,\e}- f_0)\in L^2(Q).
$$
Multiplying the above equation with $\p_t \Phi_\e(M_{s,\e})=\p_t u\in L^2(Q)$, integrating in $Q$, and using integration by parts we conclude that
$$
\int_Q \p_t u\, \D u= -\int_0^T \p_t (\tfrac{1}{2}\|\del u\|^2)= \tfrac{1}{2}\|\del \Phi_\e(M_0)\|^2 - \tfrac{1}{2}\|\del u(T)\|^2,
$$
which proves \eqref{eq:apriori2}. The detailed steps mimic its discrete counterpart in \Cref{lemma:HigherReg}.
\end{subequations}

\subsection{Finite time blow-up}\label{sec:blowup}

The model \eqref{eq:main} breaks down when $M$ reaches $1$. Henceforth, we will refer to this as \emph{blow-up}. Unlike in the case of  Dirichlet or mixed boundary conditions, this situation cannot in general be excluded for homogeneous Neumann conditions.  Whether a solution will blow-up in finite time  or not depends on the initial values $M_0,\,\vec{S}_0$.
One can construct cases when the solution will definitely blow-up in finite time. We give a simple example below.


\begin{example}[Constant initial states] Let us focus on the cellulolytic biofilm model with a single substrate \cite{eberl2017spatially}, i.e., we look at the system
\begin{subequations}\label{eq:biofilm}
\begin{align}
\p_t M&=\D\Phi(M) + f_0(M,S_1) \text{ in } Q,& M(0)&=M_0 \text{ in } \Omega,& [\del M\cdot\bm{\hat{n}}]|_{\p \Om}&=0,\\
\p_t S_1&= f_1(M,S_1) \text{ in } Q,& S_1(0)&=S_{0,1} \text{ in } \Omega.&
\end{align}
\end{subequations}
Moreover, the reaction terms are given by a non-dimensionalized version of \eqref{biofilm_model},
\begin{align}\label{eq:fgForms}
f_0(m,s)=\left (\frac{s}{1+ s} -\lambda \right ) m,\quad f_1(m,s)=-\frac{s\,m}{1+s}. 
\end{align}
For the initial and boundary values we assume that
\begin{align}
M_0\equiv \bar{M}\in (0,1), \quad S_{0,1}\equiv \bar{S}, 
\end{align}
where $\bar{M},\,\bar{S}>0$ are given constants. 
Then it is clear that the solution $(M,S_1)$ of \eqref{eq:main} remains constant in space for a given time. Hence, 
the system evolves according to the system of ODEs
\begin{align}\label{eq:explodeEq}
\p_t M= \frac{MS_1}{1+ S_1} -\lambda\,M, \quad \p_t S_1= - \frac{MS_1}{1+ S_1}\  \text{ for } t>0 \ \text{ with } (M(0),S_1(0))=(\bar{M},\bar{S}).
\end{align}
Clearly, for $\bar{M}$ close to 1, $\bar{S}$ large, and $\lambda$ small, the biomass density $M$ reaches 1 in finite time. 
\end{example}

It is possible to generalize \Cref{lemma:MaxRegSol} to provide a necessary condition for  blow-up in finite time, or a sufficient condition for $M$ to stay bounded away from $1$. This is stated in the following proposition for the single substrate ($k=1$) case.
 
\begin{proposition}[Upper and lower bounds of $(M,S_1)$]\label{pros:TwoSided}
Let \ref{prop:D}--\ref{prop:BC}  and \ref{prop:fg2} be satisfied, $k=1$ (single  substrate)  and $\G_1=\emptyset$ (homogeneous Neumann condition) in \eqref{eq:main}. Recall \ref{prop:IC}, and let $f_0(m,s)$ be increasing with respect to $s\geq 0$ for fixed $m$, and let $f_1(m,s)$ be decreasing with respect to $m\geq 0$ for fixed $s$.
Let  $(\check{M},\check{S},\hat{M},\hat{S})\in C^1(\R^+)^4$ be the solution of the ODE system 
\begin{align}\label{eq:SubSupSol}
\begin{cases}
\p_t \check{M}=f_0( \check{M},\check{S}),\qquad \p_t \check{S}=f_1( \hat{M},\check{S}),\\
\p_t \hat{M}=f_0(\hat{M},\hat{S}),\qquad \p_t \hat{S}=f_1(\check{M},\hat{S}),
\end{cases}
\end{align}
with $(\check{M},\check{S},\hat{M},\hat{S})=(\underline{M},\underline{S},\overline{M},\overline{S})$ at $t=0$. Further, assume that if $\nu_1>0$, then for $h_1$ defined in \ref{prop:BC}, $\check{S}(t)\leq h_1 \leq \hat{S}(t)$ a.e. in $\p\Om$ for all $t\in [0,T]$.
Let $(M,S)$ be the weak solution of \eqref{eq:main} in the sense of \Cref{def:WeakSol}. Then for all $t\in [0,T]$,
\begin{align}
\check{M}(t)\leq M(t)\leq \hat{M}(t)\ \text{ and }\ \check{S}(t)\leq S_1(t)\leq \hat{S}(t)\ \text{ a.e. in } \Om.
\end{align}
\end{proposition}

\begin{remark}[Assumptions in \Cref{pros:TwoSided}]
    Observe that the assumptions of \Cref{pros:TwoSided} are satisfied by the reaction terms in \eqref{biofilm_model} and \eqref{eq:fgForms} which were considered, e.g., in \cite{eberl2001new,eberl2017spatially}. Moreover, the assumption $h_1\in [\check{S}(t),\hat{S}(t)]$ a.e. in $\p\Om$ is a consistency condition that can be omitted in the case of immobile substrates ($\nu_1=0$) which occurs in the models for cellulolytic biofilms \cite{eberl2017spatially}, or  when homogeneous Neumann conditions are assumed  for $S$. 
\end{remark}

\begin{proof}
The proof generalizes the arguments in \Cref{lemma:MaxRegSol} and follows the proof of Proposition 1 of \cite{mitra2020existence}. The existence and uniqueness of the solution $(\check{M},\check{S},\hat{M},\hat{S})$ is evident from the Picard-Lindel\"of Theorem. Moreover, $f_0(m,s)$ is increasing with $s$, $f_1(m,s)$ is decreasing with $m$, along with  $\hat{M}(0)=\bar{M}\geq \underline{M}= \check{M}(0)$ and $\hat{S}(0)=\bar{S}\geq \underline{S}= \check{S}(0)$, together imply for all $t>0$, 
\begin{align}\label{eq:OrderingHatCheck}
    \hat{M}(t)\geq \check{M}(t), \quad  \text{ and } \quad \hat{S}(t)\geq \check{S}(t).
\end{align}
This follows by writing  from \eqref{eq:SubSupSol}, 
\begin{align*}
&\frac{1}{2}[\check{M}(t)-\hat{M}(t)]_+^2= \int_0^t [\check{M}-\hat{M}]_+(f_0(\check{M},\check{S})-f_0(\hat{M},\hat{S})),\\
&\frac{1}{2}[\check{S}(t)-\hat{S}(t)]_+^2= \int_0^t [\check{S}-\hat{S}]_+(f_1(\hat{M},\check{S})-f_1(\check{M},\hat{S})),
\end{align*}
for $t>0$. Then, following the manipulations in \Cref{lemma:MaxRegSol} (also repeated below), Gronwall's Lemma yields $[\check{M}(t)-\hat{M}(t)]_+=[\check{S}(t)-\hat{S}(t)]_+=0$. We omit the detailed proof for brevity.


Insert the test functions $\f=[M-\hat{M}]_+$ and $\z_1=[S_1-\hat{S}]_+$ in \eqref{eq:weak}. Observe that, $\z_1\in L^2(0,T;H^1_0(\Om))$  is a valid test function for $\nu_1>0$ since $S_1-\hat{S}=h_1 -\hat{S}\leq 0 $ on $\p\Om$. Then following the manipulations in \Cref{lemma:MaxRegSol}, one obtains from the first equation that
\begin{subequations}\label{eq:TwoSidedBound}
\begin{align}
&\int_0^T \p_t \left (\frac{1}{2} \|[M-\hat{M}]_+\|^2\right )\leq \int_0^T (f_0(M,S_1) -f_0(\hat{M},\hat{S}),[M-\hat{M}]_+)\nonumber\\
&= \int_0^T  (f_0(M,S_1)-f_0(\hat{M},S_1),[M-\hat{M}]_+) + \int_0^T (f_0(\hat{M},S_1)-f_0(\hat{M},\hat{S}),[M-\hat{M}]_+)\nonumber\\
&\overset{\ref{prop:fl}}\leq C_L \int_0^T  \|[M-\hat{M}]_+\|^2 + C_L\int_0^T ([S_1-\hat{S}]_+,[M-\hat{M}]_+)\nonumber\\
&\leq C \int_0^T [\|[M-\hat{M}]_+\|^2 + \|[S_1-\hat{S}]_+\|^2].
\end{align}
Here, we used that $f_0(\hat{M},\cdot)$ is increasing to conclude that $$
(f_0(\hat{M},S_1)-f_0(\hat{M},\hat{S}))[M-\hat{M}]_+\leq C_L [S_1-\hat{S}]_+[M-\hat{M}]_+.
$$
Similarly, from the second equation, noting that $f_1(\cdot,s)$ is decreasing for a given $s$, one obtains
\begin{align}
&\int_0^T \p_t \left (\frac{1}{2} \|[S_1-\hat{S}]_+\|^2\right )\leq \int_0^T (f_1(M,S_1) -f_1(\check{M},\hat{S}),[S_1-\hat{S}]_+)\nonumber\\
&= \int_0^T  (f_1(M,S_1)-f_1(\check{M},S_1),[S_1-\hat{S}]_+) + \int_0^T (f_1(\check{M},S_1)-f_1(\check{M},\hat{S}),[S_1-\hat{S}]_+)\nonumber\\
&\leq C_L\int_0^T ([M-\check{M}]_-,[S_1-\hat{S}]_+)+ C_L \int_0^T  \|[S_1-\hat{S}]_+\|^2 \nonumber\\
&\leq C \int_0^T [\|[M-\check{M}]_-\|^2 + \|[S_1-\hat{S}]_+\|^2].
\end{align}
\end{subequations}
Finally, inserting the test functions $\f=[M-\check{M}]_-$ and $\z_1=[S_1-\check{S}]_-$ in \eqref{eq:weak} we get analogous estimates to \eqref{eq:TwoSidedBound}. Adding these inequalities and using Gronwall's Lemma completes the proof. 
\end{proof}

\begin{remark}[Guaranteed finite time blow-up/ boundedness]
If the solution $\check{M}$ in \Cref{pros:TwoSided} reaches $1$ in finite time, then it implies that the solution of the original system $(M,S_1)$ blows up in finite time. On the other hand, if the solution $\hat{M}$ remains bounded by a constant strictly less than $1$, then $M$ does not blow up and hence, the solution $(M,S_1)$ is global-in-time. The bounds are sharp if $|\bar{M}-\underline{M}|$ and $|\bar{S}-\underline{S}|$ are small.
\end{remark}

\section{Spatial regularity of the biomass density}\label{sec:fronts}

In this section, we analyze the spatial regularity of solutions of the degenerate diffusion equation \eqref{eq:M}, i.e., we focus on the scalar equation 
\begin{align}\label{eq:Mscalar}
\p_t M =\del\cdot[D(M)\del M] + f(M,\cdot)\qquad \text{in } Q,
\end{align}
where $D:[0,1)\to [0,\infty)$ and $f:[0,\infty)\times Q\to \R$.
The regularity results we derive apply to a broad class of degenerate diffusion problems, see \Cref{rem:generality}, including the biofilm growth models \cite{eberl2001new,eberl2017spatially}.

It is well known that the degeneracy of the diffusion coefficient $D(0)=0$ causes a finite speed of propagation and sharp fronts at the interface between the regions $\{M>0\}$ and $\{M=0\}$ corresponding to steep gradients of $M$.  Despite this fact, the solution $M$ is  locally H\"older continuous. This was shown for porous medium type equations in \cite{friedman1985holder} and for equations with degenerate and singular diffusion in \cite{victor2022}. The global space-time regularity of solutions of the porous medium equation in $\R^d$ has also been studied extensively using optimal regularity theory, see \cite{guess2020optimal} and the references therein. 
Assuming homogeneous Neumann boundary conditions, in this section we show that $M$ can further inherit global spatial regularity in the more general case \eqref{eq:Mscalar}, i.e. $M\in L^2(0,T;H^r(\Om))$, where $r=1$ for $a<2$, and $r<1$ otherwise. This fact is not only mathematically intriguing but has important consequences in designing numerical tools and test functions for such problems. We now specify the assumptions on the functions $D$ and $f$.

\begin{assumption}[Assumptions on $D$ and $f$]\label{ass:D}
The diffusion coefficient satisfies \ref{prop:D}. In addition, there exists $a\in \R^+$ and a constant $C>0$ such that $$D(m)\geq C m^a$$ for all $m\in [0,1)$. The function $f:[0,\infty)\times Q\to \R$ is Lipschitz continuous with respect to the first variable, and there exists a non-negative function $f_{\max}\in \mathrm{Lip}(\R)$ such that $f(\cdot,(\bm{x},t)) \leq f_{\max}(\cdot)$. Moreover, we assume that $f(0,(\bm{x},t))\geq 0$ for all $(\bm{x},t)\in Q$.
\end{assumption}


\begin{theorem}[Global spatial regularity of $M$] 
Let $\Gamma_1=\emptyset$ (homogeneous Neumann condition). Let $M\in \W$  with $\Phi(M)=\int_0^M D\in L^2(0,T;H^1(\Om))$, and $M(0)=M_0$ (see \ref{prop:IC}) be the weak solution of \eqref{eq:Mscalar}, i.e.,
\begin{align}\label{eq:scalar_diff}
    \int_0^T \langle \f,\p_t M \rangle + \int_0^T (\del \Phi(M),\del \f)= \int_0^T (f(M,\cdot),\f),
\end{align} 
for all $\f\in L^2(0,T;H^1(\Om))$. Then under the Assumption \ref{ass:D}, $M\in L^2(0,T;H^r(\Om))$ for 
\begin{enumerate}[label=(\alph*)]
\item $r=1$ if either $a<2$ or $\underline{M}=\mathrm{ess}\inf\{M_0\}>0$.
\item all $r<2/a$, if $a\geq 2$ and $\underline{M}=0$.
\end{enumerate}\label{theo:fronts}
\end{theorem}

\begin{remark}[Generality of Theorem \ref{theo:fronts}] \label{rem:generality}
Theorem \ref{theo:fronts} applies to the solution $M$ of the coupled system \eqref{eq:main} under the conditions \ref{prop:D}--\ref{prop:BC} and \ref{prop:fg2}. Since the spatial irregularity of $M$ stems from the degeneracy at $M=0$, our regularity results also cover diffusion coefficients $D$ that are degenerate but non-singular, for instance, porous medium type equations. In this case, the additional assumption that $f$ is bounded by $f_{\max}$ can be omitted as solutions are not required to take values in $[0,1)$.
\end{remark}

\begin{remark}[Assumptions on the boundary conditions in Theorem \ref{theo:fronts}]
To simplify notations \Cref{theo:fronts} is stated for homogeneous Neumann boundary conditions. However, the result remains valid for Dirichlet or mixed boundary conditions provided that $\Phi(M)=\Phi(h^e_0)$ at $\Gamma_1$ and the functions $\Psi_\e(h^e_0)\in H^1(\Om)$ are uniformly bounded with respect to $\e\in (0,1)$, where $\Psi_\e$ is introduced in \eqref{eq:DefPsi} and $h^e_0$ in \ref{prop:BC}.
\end{remark}

The rest of this section is dedicated to the proof of  \Cref{theo:fronts}. The main idea behind the proof is to use a test function $\f$ of the form $M^{-\a}$ ($\a>0$) in \eqref{eq:scalar_diff}. However, $\varphi$ might not be a valid test function due to $M$ not being sufficiently regular, and $M^{-\a}$ having a singularity at $0$. To resolve this, we will construct a modified function that is admissible.

\subsection{Some auxiliary functions}

As in the proof of \Cref{theo:fronts} we consider the regularized problem introduced in \Cref{lemma:ExistRegSol}. The function $\Phi_\e$ is taken as in \eqref{eq:PhiepsDef}. For a given constant $\a>0$, we further introduce the $C^1(\R)$ function 
\begin{align}\label{eq:DefPsi}
\Psi_\e(m):= \int_{1}^m \frac{\dd \vr }{\min\{\max\{\e,\vr^\a\},1\}}.
\end{align}
Note that
\begin{align}\label{eq:propPsi}
\Psi'_{\e}\ge 0 \quad \text{ and }\quad  \Psi_{\e}(m)< 0 \quad \text{ for } m<1.
\end{align}

\begin{lemma}[Growth of $\Psi_\e$] 
For a given $\a> 0$ and $\e\in (0,1)$, let $\Psi_\e$ be defined as in \eqref{eq:DefPsi}. Then, the following estimate holds,
\begin{align}
|m\Psi_\e(m)|\lesssim 1 + \int_{1}^m \Psi_\e\qquad  \text{ for all } m\geq 0.\label{eq:PsiIneq}
\end{align}
\end{lemma}
\begin{proof}
\textbf{Case 1 ($1<m$):}  For $m>1$, $\Psi_\e(m)=m-1$, and the inequality can be verified directly. 

\textbf{Case 2 ($\e^{\frac{1}{\a}}\leq m\leq 1$):} If $\e^{\frac{1}{\a}}\leq m\leq 1$ and $\a\not=1$, we have
\begin{subequations}\label{eq:Case1Am}
\begin{align}
|m \,\Psi_{\e}(m)|\leq \left|m \int^{m}_1 \frac{1}{\vr^\a}\right| d\rho|= 
\left |m\left (\frac{m^{1-\a}-1}{1-\a}\right ) \right | \lesssim |m^{2-\a}-m|.
\end{align}
Observe that, if $\a\leq  2$ then the right hand side is bounded since $m\leq 1,$ and \eqref{eq:PsiIneq} definitely holds. The case when $\a=1$ can also be handled rather easily since it yields $\Psi_\e(m)=\log(m)$. The  interesting case is when $\a>2$. Then, we can estimate  the right-hand side of \eqref{eq:PsiIneq} as follows,
\begin{align}
1+ \int_{1}^m \Psi_{\e}&= 1+ \int_{1}^{m}\int_{1}^s \frac{1}{\vr^\a}d\rho d s = 1+  \int_{1}^{m} \frac{s^{1-\a}-1}{\a-1}\,\dd s\gtrsim \int_1^m s^{1-\a}  \gtrsim  m^{2-\a} -1. 
\end{align}
\end{subequations}
Combining \eqref{eq:Case1Am} and noting that $m<1$, we have \eqref{eq:PsiIneq} for this case.

\textbf{Case 3 ($0\leq m<\e^{\frac{1}{\a}}$):}
We only focus on $\a>2$ since the case $\a\leq 2$ can be shown exactly as in Case 2. We observe that 
\begin{subequations}\label{eq:Case2Am}
\begin{align}
|m \,\Psi_{\e}(m)|=& \left |m \Psi_{\e}(\e^{\frac{1}{\a}}) + m\smallint^{m}_{\e^{\frac{1}{\a}}} \tfrac{\dd \vr}{\max(\e,\vr^\a)}  \right |=\left |m \Psi_{\e}(\e^{\frac{1}{\a}}) + m\smallint^{m}_{\e^{\frac{1}{\a}}} \tfrac{1}{\e}  \right |\nonumber\\
\leq & |m \Psi_{\e}(\e^{\frac{1}{\a}})| + \frac{m(\e^{\frac{1}{\a}}-m)}{\e}
\leq |\e^{\frac{1}{\a}} \Psi_{\e}(\e^{\frac{1}{\a}})| + \e^{\frac{1}{\a}-1}(\e^{\frac{1}{\a}}-m).
\end{align}
From Case 2 we conclude that $|\e^{\frac{1}{\a}} \Psi_{\e}(\e^{\frac{1}{\a}})|\lesssim 1+ \int^{\e^{\frac{1}{\a}}}_{1} \Psi_{\e}$, and we obtain
\begin{align}
1+ \int_{1}^m \Psi_{\e}&= 1+ \int_{1}^{\e^{\frac{1}{\a}}} \Psi_{\e} + \int^{m}_{\e^{\frac{1}{\a}}} \int_{1}\tfrac{\dd \vr}{\max(\e,\vr^\a)}= 1+ \int_{1}^{\e^{\frac{1}{\a}}} \Psi_{\e} + \int^{\e^{\frac{1}{\a}}}_{m} \int^{1}\tfrac{\dd \vr}{\max(\e,\vr^\a)}\nonumber\\
& \geq 1 + \int_{1}^{\e^{\frac{1}{\a}}} \Psi_{\e} + \int_{m}^{\e^{\frac{1}{\a}}} \int_{\e^{\frac{1}{\a}}}^{1}\tfrac{\dd \vr}{\max(\e,\vr^\a)}\gtrsim |\e^{\frac{1}{\a}} \Psi_{\e}(\e^{\frac{1}{\a}})| + (\e^{\frac{1}{\a}}-m)\int_{\e^{\frac{1}{\a}}}^{1}\tfrac{\dd \vr}{\vr^\a}\nonumber\\
&\gtrsim |\e^{\frac{1}{\a}} \Psi_{\e}(\e^{\frac{1}{\a}})| + \e^{\frac{1}{\a}-1}(\e^{\frac{1}{\a}}-m)- (\e^{\frac{1}{\a}}-m).
  \end{align}  
Hence, combining again \eqref{eq:Case2Am} we have \eqref{eq:PsiIneq}.
\end{subequations}
\end{proof}

\subsection{Boundedness of $M$ in $L^2(0,T;H^r(\Om))$}
To prove \Cref{theo:fronts} we first show the following lemma. 

\begin{lemma}[An estimate for the regularized solutions]\label{lemma:UniBoundW} 
Let Assumption \ref{ass:D} hold and $\Gamma_1=\emptyset$. For $\e\in (0,1)$ and $\a>0$, let $\Phi_\e$ and $\Psi_\e$ be defined by \eqref{eq:PhiepsDef}  and \eqref{eq:DefPsi} respectively. Let $M_{\e}\in \X$ satisfy $M_{\e}(0)=M_0$ and \begin{align*}
        \int_0^T \langle \f,\p_t M_\e \rangle + \int_0^T (\del \Phi_\e(M_\e),\del \f)= \int_0^T (f(M_\e,\cdot),\f),
\end{align*}
for all $\f\in L^2(0,T;H^1(\Om))$. Then, we have 
\begin{align}
\int_0^T \int_{\Om} \min\{M_\e^{a-\a},1\}|\del M_\e|^2\lesssim 1 + \int_\Om \int_{1}^{M_0} \Psi_{\e}.\label{eq:Mineq}
\end{align}
\end{lemma}

\begin{proof} First, we show that the following estimate holds,
\begin{align}\label{eq:PhiPsi}
\Phi'_\e(m)\, \Psi'_\e(m)\gtrsim \min\{1, m^{a-\a}\}\qquad  \text{ for all } m\geq 0.
\end{align}
We distinguish several cases. 
\textbf{Case 1: ${\Phi_\e}'(m)=D(m)$.} 
This implies that $\varepsilon\leq {\Phi_\e}'(m)=D(m)\leq \frac{1}{\varepsilon}$ and hence,  $m<1$. 
If  ${\Psi_\e}'(m)= m^{-\a}$, then the result follows from Assumption \ref{ass:D}. 
If ${\Psi_\e}'(m)= \frac{1}{\varepsilon}$, then $\Phi'_\e(m)\, \Psi'_\e(m)=D(m)/\e>1$.  
Finally, if ${\Psi_\e}'(m)= 1$, then $m\geq 1$ which is excluded. 

\textbf{Case 2: ${\Phi_\e}'(m)=\e$.} Consequently, $m<1$. The definition of $\Phi_\e$ in \eqref{eq:PhiepsDef} implies that $\e
\geq D(m)\gtrsim m^a$, where the last inequality holds by Assumption \ref{ass:D}. Hence, for $\Psi'_\e=\e^{-1}$ the product $\Phi'_\e\, \Psi'_\e=1$ and for $\Psi'_\e=m^{-\a}$ we have  $\Phi'_\e\, \Psi'_\e\gtrsim m^{a-\a}$. 

\textbf{Case 3: ${\Phi_\e}'=\frac{1}{\varepsilon}$.} This case  follows similarly. 

Inserting the test function $\f=\Psi_{\e}(M_{\e})$ in \eqref{eq:reg}, the first term becomes
\begin{subequations}\label{eq:TheL2H1part}
\begin{align}
\int_0^T \langle \p_t M_{\e}, \Psi_{\e}(M_{\e})\rangle= \int_\Om \int_{1}^{M_{\e}(T)}  \Psi_{\e} - \int_\Om \int_{1}^{M_0} \Psi_{\e}.
\end{align}
 The second term of \eqref{eq:reg} gives
\begin{align}
\int_0^T (\del \Phi_\e(M_\e),\del \Psi_{\e}(M_{\e}))&=
\int_0^T \int_{\Om} \Phi'_\e(M_\e)\,\Psi'_\e(M_\e) |\del M_{\e}|^2\nonumber\\
&\overset{\eqref{eq:PhiPsi}}\gtrsim \int_0^T  \int_{\Om} \min\{M_\e^{a-\a},1\}|\del M_\e|^2.
\end{align}
Finally, the third term of \eqref{eq:reg} yields using $f(0,\cdot)\geq 0$ and \eqref{eq:propPsi} that  
\begin{align}
&\int_0^T ( f(M_{\e},\cdot), \Psi_{\e}(M_{\e}))=\int_0^T ( f(M_{\e},\cdot)-f(0,\cdot), \Psi_{\e}(M_{\e}))+
\int_0^T (f(0,\cdot), \Psi_{\e}(M_{\e}))\nonumber\\
&\quad  \overset{\ref{prop:fl}} \leq C_L \int_0^T \int_\Om | M_{\e} \Psi_{\e}(M_{\e})| \overset{f(0,\cdot)\geq 0}+
\int_0^T (f(0,\cdot), [\Psi_{\e}(M_{\e})]_+)\nonumber\\
&\quad \lesssim  \int_0^T \int_\Om | M_{\e} \Psi_{\e}(M_{\e})| \overset{\eqref{eq:PsiIneq}} \lesssim \int_0^T \int_\Om [1+ \int_{1}^{M_{\e}}  \Psi_{\e}].
\end{align}
\end{subequations}
In the above, noting that $\Psi_{\e}(m)> 0$ only when $m>1$, we estimated $f(0,\cdot) [\Psi_{\e}(M_{\e})]_+\leq f_{\max}(0) |M_\e \Psi_{\e}(M_{\e}) |$. Combining the inequalities \eqref{eq:TheL2H1part} we have
\begin{align}
\int_\Om \int_{1}^{M_{\e}(T)}  \Psi_{\e} +\int_0^T  \int_{\Om} \min\{M_\e^{a-\a},1\}|\del M_\e|^2 \lesssim 1 +\int_{\Om}\int_{1}^{M_0} \Psi_{\e}+  \int_0^T \int_\Om \int_{1}^{M_\e} \Psi_{\e}.
\end{align}
Using Gronwall's Lemma \eqref{eq:Gronwall} the estimate \eqref{eq:Mineq} follows. 
\end{proof}

To conclude the proof of \Cref{theo:fronts} from \eqref{eq:Mineq}, we need the following lemma.
For its proof we refer to Lemma 1.3 and Lemma B.1 of \cite{sonner2012systems}.

\begin{lemma}[Property of $H^r(\Om)$]\label{lemma:Hr}
If $u^\g\in H^1(\Om)$ for some $\g>1$ then $u\in H^r(\Om)$ for all $r\in (0,\g^{-1}]$.
\end{lemma}

\begin{proof}[\textbf{Proof of \Cref{theo:fronts}}]
\textbf{Case 1 ($\underline{M}=\mathrm{ess}\inf \{M_0\}>0$):} In this case, taking $\e_1^{\frac{1}{\alpha}}<\underline{M}$ we conclude by \eqref{eq:DefPsi}
that
\[
\int_{\Om} \int_1^{M_0} \Psi_\e \quad \text{ is uniformly bounded for all } \e \leq \e_1.
\]
Hence, taking $\a=a$ in Lemma \ref{lemma:UniBoundW} provides a uniform bound on $\int_0^T \|\del M_\e\|^2$. Moreover, $\|M_\e\|_{L^\infty(0,T;L^\infty(\Om))}$ is bounded by \Cref{lemma:MaxRegSol}. Hence, $M_\e$ is uniformly bounded in $L^2(0,T;H^1(\Om))$. Passing to the limit $\e\to 0$, the convergence of $M_\e$ to a unique $M\in \W$ follows from \Cref{lemma:ExistUnRegSol} 
  (see \eqref{eq:convergencePhiM}). Consequently, the uniform bound implies that $M\in L^2(0,T;H^1(\Om))$.

\textbf{Case 2 ($a<2$):} In this case, put $\a=a$. Then, passing the limit $\e\to 0$ on the right hand side of \eqref{eq:Mineq} one has 
$$
\lim\limits_{\e\searrow 0}\int_{\Om}\int_1^{M_0} \Psi_\e\lesssim \int_{\Om}\int_{M_0}^1 (1-\vr^{1-a})\,\dd\vr\lesssim 1.
$$
Here we used that $M^{2-a}_0\leq 1$ a.e. in $\Om$ by assumption \ref{prop:IC}. Hence, we again obtain a uniform bound on $\int_0^T \|\del M_\e\|^2$ and consequently, $M\in L^2(0,T;H^1(\Om))$.

\textbf{Case 3 ($a\geq 2$):} We set $\a=2-\d$ for sufficiently small $\d>0$. Then the previous case and \eqref{eq:Mineq} gives that $(M_\e)^{1+ \frac{a-\a}{2}}\in L^2(0,T;H^1(\Om))$ and it is uniformly bounded. From \Cref{lemma:Hr} it follows that $M_\e\in L^2(0,T;H^r(\Om))$ for 
$$
r\leq \frac{2}{2+ a-\a}=\frac{2}{a+\d}\qquad \text{and sufficiently small } \d>0.
$$
This concludes the proof.
\end{proof}


 \subsection*{Acknowledgements}
 K. Mitra and S. Sonner would like to thank the Nederlandse Organisatie voor  Wetenschappelijk
Onderzoek (NWO) for their support through the Grant OCENW.KLEIN.358. K. Mitra was additionally
supported by Fonds voor Wetenschappelijk Onderzoek (FWO) through the Junior Postdoctoral Fellowship during the completion of this work.

\bibliographystyle{plain}

\end{document}